\documentclass[12pt,reqno,final]{amsart}
\usepackage[margin=1in]{geometry}
\usepackage{amsmath,amssymb,amsthm,color,url}
 \numberwithin{figure}{section}
 \usepackage{graphicx}

\DeclareMathOperator{\SO}{SO}

\DeclareMathOperator{\lt}{lt}
\DeclareMathOperator{\lm}{lm}
\DeclareMathOperator{\lcm}{lcm}

\newcommand{\R}{\mathbb{R}}

\numberwithin{equation}{section}
\newtheorem{theorem}{Theorem}[section]
\newtheorem{proposition}[theorem]{Proposition}
\newtheorem{lemma}[theorem]{Lemma}

\newtheorem{corollary}[theorem]{Corollary}

\newtheorem*{question*}{Question}

\theoremstyle{definition}

\newtheorem*{definition*}{Definition}

\theoremstyle{remark}
\newtheorem{remark}[theorem]{Remark}
\newtheorem*{remark*}{Remark}

\usepackage[normalem]{ulem}

\title{On integer distance sets}
\author{Rachel Greenfeld}
\address{Department of Mathematics, Northwestern University, 2033 Sheridan Road, Evanston, IL 60208, USA}
\email{rgreenfeld@northwestern.edu}
\author{Marina Iliopoulou}
\address{Department of Mathematics, National and Kapodistrian University of Athens, Panepistimiopolis 157 84, Athens, Greece}
\email{miliopoulou@math.uoa.gr}
\author{Sarah Peluse}
\address{Department of Mathematics, Stanford University, 450 Jane Stanford Way, Building 380, Stanford, CA 94305, USA}
\email{speluse@stanford.edu}

\begin{document}

\begin{abstract}
  We develop a new approach to address some classical questions
  concerning the size and structure of integer distance sets. Our main
  result is that any integer distance set in the Euclidean plane is either very sparse or has all but an exceedingly small proportion of its points lying on a single line or circle. From this, we deduce a near-optimal lower bound on the diameter of any non-collinear integer distance set of size $n$ and a strong upper bound on the size of any integer distance set in $[-N,N]^2$ with no three points on a line and no four points on a circle.
  \end{abstract}
  \maketitle

\section{Introduction}\label{sec:intro}
A subset $S\subset\mathbf{R}^2$ of the plane is an \textit{integer distance set} if the
Euclidean distance between every pair of points in $S$ is an integer. In 1945, Anning and
Erd\H{o}s~\cite{AnningErdos1945,Erdos1945} showed that any infinite integer distance set
must be contained in a line, but also that there exist non-collinear integer distance sets
of arbitrarily large finite size. They constructed two infinite families of arbitrarily
large non-collinear integer distance sets: one of concyclic sets and one of sets in which
all but one point are collinear. It turns out that all so-far-known integer distance sets, such
as those in~\cite{HarborthKemnitzMoller1993,Huff1948,Peeples1954}, are of a similar
special form: they have all but up to four of their points lying on a single line or
circle.

In this paper, we develop a new approach to the study of integer distance sets that enables us to prove a structure theorem partially explaining this phenomenon, showing that any integer distance set in $[-N,N]^2$ must either be polylogarithmically small or have the vast majority of its points lying on a single line or circle.

\begin{theorem}[Structure theorem]\label{thm:main}
  Let $S\subset[-N,N]^2$ be an integer distance set. 
Then, either
$\left|S \right|=O\left((\log{N})^{O(1)}\right)$,
or else there exists a line or circle $C\subset\mathbf{R}^2$ such that
\begin{equation*}
    \left|S\setminus C\right|=O\left((\log\log{N})^{2}\right).
\end{equation*}
\end{theorem}

\begin{remark} This result can be extended to subsets of $\mathbf{R}^2$ with pairwise rational distances of height at most $N$. We leave the details to the interested reader. \end{remark}

Theorem~\ref{thm:main} is the first-ever general structure theorem for integer distance sets, and we use it to address two classical questions about integer distance sets.

\subsection{Erd\H{o}s's integer distance set problem} 
The first question we address is the question of how large an integer distance set can be if it has no three points on a line and no four points on a circle (see, for example,~\cite[Chapter 13]{Eppstein2018}). This question is often attributed to Erd\H{o}s, though he says in~\cite{Erdos1987} that it is an ``old problem whose origin I cannot trace''. Erd\H{o}s still posed it on numerous occasions (for instance, in~\cite{Erdos1983}), and, in particular, asked whether such a set consisting of seven points exists in the plane. In 2008, Kreisel and Kurz~\cite{KreiselKurz2008} indeed found seven points in the plane with no three of them on a line, no four of them on a circle, and with pairwise integer distances. No larger such set has been found since.

It is an immediate consequence of Theorem~\ref{thm:main} that integer distance sets with
no three points on a line and no four points on a circle must be very sparse.
\begin{corollary}[New upper bound in Erd\H{o}s's integer distance set problem]\label{cor:mainno4}
  Let $S\subset[-N,N]^2$ be an integer distance set with no three points on a line and
  no four points on a circle. Then,
  \begin{equation*}
    \left| S\right|=O\left( (\log{N})^{O(1)}\right).
  \end{equation*}
\end{corollary}
The previous best known upper bound for the size of integer distance sets in $[-N,N]^2$
with no three points on a line and no four points on a circle was $O(N)$, which follows
from work of~Solymosi~\cite{Solymosi2003}. It could even be the case that there is a
uniform upper bound on the size of any integer distance set with no three points on a line
and no four points on a circle. Indeed, Ascher, Braune, and
Turchet~\cite{AscherBrauneTurchet2020} proved this conditionally on Lang's conjecture.

\begin{remark} 
    Interestingly,  Erd\H{o}s's integer distance set problem is related to Fuglede's problem on the maximal size of frequency sets $\Lambda\subset \mathbf{R}^2$ whose corresponding system  \[E(\Lambda)=\left\{e^{2\pi i  \lambda\cdot x}\colon \lambda\in \Lambda\right\}\] of exponential functions is orthogonal in $L^2(D)$, where $D=\{x\in \mathbf{R}^2\colon \|x\| < 1\}$ is the unit disk (see \cite{IosevichRudnev2003}, where this relation is explored).
     Fuglede proved that any such $\Lambda$ must be finite \cite{FUGLEDE1974101, FUGLEDE2001267} (see also \cite{IosevichKatzPedersen1999, IosevichKatzTao2001, IosevichRudnev2003} for related results). 

Note that the system $E(\Lambda)$ is orthogonal in $L^2(D)$ if and only if
\[\int_D e^{2\pi i  (\lambda-\lambda')\cdot x} d x =\widehat{\mathbf{1}}_{D}(\lambda'-\lambda)=0\]
 for any pair of distinct points $\lambda,\lambda'$ in $\Lambda$, and recall that $\xi\in \mathbf{R}^2$ is a zero of $\widehat{\mathbf{1}}_D$ if and only if $2\pi \|\xi\|$ is a zero of the Bessel function $J_1$.  Therefore, the system $E(\Lambda)$ is orthogonal in $L^2(D)$ if and only if the distance set of $\Lambda$,
 \begin{equation}\label{eq:distLambda}
     \left\{\|\lambda'-\lambda\| \colon \lambda,\lambda'\in \Lambda\text{ and } \lambda\neq \lambda' \right\},
 \end{equation}
is contained in the zero-set of the function $J_1(2\pi \: \cdot)$.  Using the asymptotic expansion  \cite{AbramowitzStegun64} of the function $J_1(2\pi \: \cdot)$, we have that its positive zeros, and hence the distances \eqref{eq:distLambda}, are of the form 
\[\frac{n}{2}+\frac{1}{8} +O\left(\frac{1}{n}\right),\quad n\in \mathbf{N}.\] 
Since the set $\left\{\frac{n}{2}+\frac{1}{8}\colon n\in\mathbf{Z}\right\}$ is not closed under addition, this then imposes (by, e.g., \cite[Lemma~1]{IosevichKolountzakis2013} and Ptolemy's theorem) geometric constraints on the frequency set $\Lambda$ that are very similar to the constraints in Erd\H{o}s's integer distance set problem, i.e., (approximate) lines and circles are excluded. 
 Using these constraints, bounds on the size of $\Lambda\cap [-N,N]^2$ were established in \cite{IosevichKolountzakis2013,zakharov2024setsorthogonalexponentialsdisk}, with the current record of $O_\epsilon(N^{3/5+\epsilon})$ due to Zakharov in~\cite{zakharov2024setsorthogonalexponentialsdisk}.

 Fuglede conjectured \cite{FUGLEDE1974101} that any frequency set $\Lambda\subset \mathbf{R}^2$ that gives rise to an orthogonal system $E(\Lambda)$ in $L^2(D)$ cannot contain more than three points. Indeed, if $|\Lambda|\geq 4$ for such a set $\Lambda$, then there would be algebraic relations among the distances \eqref{eq:distLambda}, since the $5\times 5$ Cayley--Menger determinant of any planar four-point simplex vanishes.
 This, in turn, would imply non-trivial algebraic relations between zeros of $J_1$, which 
  contradicts the expectation that these zeros are algebraically independent.
  Nevertheless, since the non-existence of algebraic relations among zeros of $J_1$ 
  is yet an open problem in transcendental number theory, this conjecture of Fuglede remains open.
 
 It is plausible, however, that one could significantly improve upon Zakharov's bound by adapting the machinery we introduce in this work, namely, encoding the points of $\Lambda$ as lattice points on an analytic manifold\footnote{For instance, let $\varphi\colon \R\to \R$ be an analytic function whose integer points are the zeros of the Bessel function $J_1(2\pi \: \cdot)$, and let $\mathcal{CM}\colon  \R^6\to \R$ be the $5\times 5$ Cayley--Menger determinant. Fix three distinct points $P_1,P_2,P_3\in \Lambda$. Then, for any point $P\in \Lambda\setminus \{P_1,P_2,P_3\}$, the simplex $(P,P_1,P_2,P_3)$ is assigned to an integer point $(\varphi^{-1})^{\otimes 6}\left(\|P_1-P_2\|,\|P_1-P_3\|,\|P_1-P_4\|,\|P_2-P_3\|,\|P_2-P_4\|,\|P_3-P_4\|\right)$ on the manifold $\mathcal{CM}(\varphi^{\otimes 6})=0$.} and then using the Pila--Wilkie theorem \cite{PilaWilkie2006}. 
\end{remark}

\subsection{Sharp diameter bound}
The second question we address is that of estimating the minimum diameter of a
non-collinear integer distance set of size $n$ (see, for example, Problem 8
of~\cite[Section 5.11]{BrassMoserPach2005}). It is known, thanks to work of
Solymosi~\cite{Solymosi2003} and Harborth, Kemnitz, and
M\"oller~\cite{HarborthKemnitzMoller1993}, that this minimum diameter has
order of magnitude at least $n$ and at most $n^{O(\log\log{n})}$. Theorem~\ref{thm:main} reduces the problem of proving a lower bound for the diameter to
the problem of proving an upper bound for the maximum number of collinear and concyclic
points a non-collinear integer distance set in $[-N,N]^2$ can contain. The
paper~\cite{KurzWassermann2011} of Kurz and Wassermann implicitly contains such a bound
for the number of collinear points.
\begin{proposition}[\cite{KurzWassermann2011}]\label{prop:lines}
   Let $S\subset[-N,N]^2$ be a non-collinear integer distance set. Then, for every
   line $\ell\subset\mathbf{R}^2$, we have
   \begin{equation*}
     \left| S\cap\ell\right|=O\left( N^{O(1/\log\log{N})}\right).
   \end{equation*}
\end{proposition}
Bat-Ochir~\cite{Bat-Ochir2019} proved an upper bound of the same strength for circles contained in $[-N,N]^2$ with
special radii, which we extend to all circles.
\begin{proposition}\label{prop:main}
  Let $C\subset\mathbf{R}^2$ be a circle and $S\subset C$ be an integer distance set. Then,
  \begin{equation*}
    \left| S\cap[-N,N]^2 \right|=O\left(N^{O(1/\log\log{N})}\right).
  \end{equation*}
\end{proposition}
Combining Theorem~\ref{thm:main} with Propositions~\ref{prop:lines} and~\ref{prop:main},
we improve Solymosi's lower bound~\cite{Solymosi2003} to obtain almost tight bounds on the smallest possible diameter of a non-collinear integer
distance set of size $n$.
\begin{corollary}[New diameter lower bound]\label{cor:mainno3} 
  Any non-collinear integer distance set of size $n$ has diameter at least
  \begin{equation*}
  \Omega\left(n^{\Omega(\log\log{n})}\right).
  \end{equation*}
\end{corollary}

It is possible to make all implied constants in the proofs of
Propositions~\ref{prop:lines} and~\ref{prop:main} completely explicit, and thus to compute an explicit $c>0$ such that the minimum diameter of a non-collinear integer
distance set of size $n$ has order of magnitude at least $n^{c\log\log{n}}$. On the other hand, one can also compute, using the
constructions in~\cite{HarborthKemnitzMoller1993}, an explicit constant $c'>c$ such that the minimum diameter of a non-collinear integer distance set
of size $n$ has order of magnitude at most $n^{c'\log\log{n}}$. We can improve these
constructions a bit to obtain a $c'$ slightly smaller than the best known one currently in the literature,
but there is still a gap between this and the largest possible value of $c$ produced by
our arguments. It is, therefore, still an open problem to determine the exact asymptotics of the
minimum diameter of a non-collinear integer distance set of size $n$. Kurz and Wassermann~\cite{KurzWassermann2011} conjectured that,
for $n$ sufficiently large, an integer distance set of size $n$ with minimum diameter has
all but one point collinear. This problem is also still open. Finally, the famous question
of Erd\H{o}s and Ulam of whether there exists a dense rational distance set in the plane
is still unresolved, though the answer is known to be negative conditionally on the $abc$
conjecture~\cite{Pasten2017} or on Lang's conjecture~\cite{Shaffaf2018,Tao2014}.

\subsection{Contrast with prior approaches.} Shortly after his original paper with Anning proving that any non-collinear integer distance set $S$ is finite, Erd\H{o}s gave a very short and elegant second proof of this result~\cite{Erdos1945}, which we now reproduce in its entirety: If $P_1$, $P_2$, $P_3$ are three distinct non-collinear points in $S$, then the fact that every point $x\in S$ has integer distance from each of $P_1$, $P_2$, and $P_3$ implies that
\begin{equation*}
    \big|\| x-P_i\|-\|x-P_{i+1}\| \big|\in\left\{0,1,\dots,\|P_i-P_{i+1}\|\right\},
\end{equation*}
for both $i=1$ and $i=2$. In other words,
\begin{equation*}
S\subset\bigcup_{\substack{H_1\in\mathcal{H}_1 \\ H_2\in\mathcal{H}_2}}(H_1\cap H_2),
\end{equation*}
where, for each $i=1,2$, $\mathcal{H}_i$ is a family of $\|P_i-P_{i+1}\|+1$ hyperbolas with foci $P_i$ and $P_{i+1}$. Since the points $P_1, P_2, P_3$ are not collinear, a hyperbola from $\mathcal{H}_1$ can intersect a hyperbola from $\mathcal{H}_2$ in at most four points. Therefore, $S$ must be finite.

All previous progress towards Corollaries~\ref{cor:mainno4} and~\ref{cor:mainno3} relied heavily on this argument of Erd\H{o}s, which is one of many applications of ``low degree algebraic geometry'' to extremal combinatorics and discrete geometry. It appears that the utility of low degree algebraic geometry in understanding the size and structure of integer distance sets is limited, however, as can be seen from the lack of progress beyond Solymosi's work~\cite{Solymosi2003}. Our new innovation is to use ``intermediate degree algebraic geometry'' to study integer distance sets; we expect this general method to have further applications in combinatorial and discrete geometry.

\subsection{Outline of our method.} 
Our technique may be viewed as a new polynomial method to count points that, after appropriate transformations, are rational points of controlled height satisfying additional conditions of an algebraic nature.

To motivate our method, let $S$ be a set of rational points of height at most $H$. Then, one strategy to control $|S|$ from above would be to show that $S$ lies on an irreducible algebraic curve $\gamma$, defined over $\mathbf{Q}$, of some \textit{intermediate} degree $d$ (or, equivalently, on a small number of such curves that are irreducible). Roughly speaking, if $d$ is too low or too high, then $\gamma$ may contain  many rational points: it is either not curved enough (e.g., it is a line), or it is too curved. There is, however, a small range of intermediate degrees $d$ (depending on $H$) for which there can be very few rational points of height at most $H$ on $\gamma$. The exact bounds follow from the analysis in~\cite{CastryckCluckersDittmannNguyen2020}, as recorded in Theorem~\ref{thm:curvecount} below. In fact, the degrees $d$ that would give optimal bounds on $|S|$ would be of magnitude $(\log H)^C$, implying  that $|S|=O((\log H)^{O(1)})$. 

Inspired by this, our approach can be outlined as follows:
\begin{enumerate}
    \item We show that any integer distance set $S$ in $[-N, N]^2$ can be transformed into a set $\widetilde{S}=LS$ of rational points of height $O(N^2)$ via an appropriate invertible linear transformation $L$ of $\mathbf{C}^2$ (see Lemma~\ref{lem:lattice} below).
    \item According to the above analysis, we would ideally like to show that if $S$ is not almost completely contained in a single line or circle, then $\widetilde{S}$ lies in the union of $O((\log N)^{O(1)})$ irreducible algebraic curves, defined over $\mathbf{Q}$, each of degree roughly $(\log N)^C$ (for some constant $C$). 
    \item To achieve this, we need to search for algebraic conditions that $\widetilde{S}$ satisfies. In particular, we fix $k\asymp \log \log N$ points $P_1,\ldots, P_k\in S$ (if, of course, $k$ points exist in $S$; note that these $k$ points are considerably more than the three points fixed in Erd\H{o}s's argument). To each $(\widetilde{x},\widetilde{y})\in \widetilde{S}$ (which corresponds to some unique $(x,y)\in S$), we associate the $k$-tuple $(d_1,\ldots, d_k)$ of the distances of $(x,y)$ from each of the $P_1,\ldots, P_k$. More precisely, we do this on the whole of $\mathbf{C}^2$ to form a complex algebraic variety: to each $(\widetilde{x},\widetilde{y})=L(x,y)\in\mathbf{C}^2$, we associate (generically $2^k$ such) $k$-tuples $(d_1,\ldots, d_k)$ with $d_i^2=\|(x,y)-P_i\|^2$ for all $i=1,\dots,k$, where $\|(z,w)\|^2=z^2+w^2$ for $z,w\in\mathbf{C}$.
    \item We prove that the set $X_k$ of all such $(\widetilde{x}, \widetilde{y}, d_1,\ldots, d_k)$ (as $(\widetilde{x},\widetilde{y})$ ranges over $\mathbf{C}^2$) is an irreducible algebraic surface, defined over $\mathbf{Q}$, of degree roughly $(\log N)^C$. Since $S$ is an integer distance set, the set $\widetilde{S}'$ of all lifted points $(\widetilde{x},\widetilde{y},d_1,\ldots,d_k)$ of $\widetilde{S}$ on $X_k$ consists of rational points of height $O(N^2)$. Therefore, we are close to the aim in (2): we have transformed our points into rational points of the correct height inside an irreducible algebraic variety of the correct degree. However, this variety is a surface, rather than a curve.
    \item To tackle this issue, we begin by covering $\widetilde{S}'$ by $O((\log N)^{O(1)})$ irreducible algebraic curves $\widetilde{\gamma}$, defined over $\mathbf{Q}$, of degree $O((\log N)^C)$, contained in the surface $X_k$. The existence of such curves is ensured as a consequence of the determinant method from analytic number theory (see Theorem~\ref{thm:surfacecount}). Projecting the first two coordinates to $\mathbf{C}^2$, we are now closer to the aim in (2), but not quite there: we have covered $\widetilde{S}$ by an acceptable number of irreducible algebraic curves $\widetilde{\gamma}$ defined over $\mathbf{Q}$. However, these curves may have very low degree relative to the desired $(\log N)^C$.
    \item We resolve this issue by utilizing the freedom we have to choose the points $P_1,\ldots, P_k\in S$ we fix at the start of the argument. Each choice leads to a different $X_k$ with properties as in (4). We prove that, if $L^{-1}\widetilde{\gamma}$ is not a line or a circle, or if it is but there are enough points in $S$ outside it, then there is a choice of initial points $P_1,\ldots, P_k$ such that the lift of $\widetilde{\gamma}$ on the corresponding $X_k$ is an irreducible curve of the correct degree $(\log N)^C$. Verifying that these lifts are irreducible is the crux of our argument.
\end{enumerate}

We will give a more detailed outline of our argument in Section~\ref{sec:setup}.

\subsection{Comparison with the polynomial method in incidence geometry.} In 2008, Dvir solved the finite field Kakeya problem~\cite{Dvir2009}, introducing, for the first time, the polynomial method into incidence geometry. This revealed a deep algebraic nature underlying incidence geometric problems. A refinement of Dvir's method, called \emph{polynomial partitioning}, was developed soon after by Guth and Katz in their solution of the Erd\H{o}s distinct distances problem in the plane~\cite{GuthKatz2015}. Polynomial partitioning has since induced a revolution in incidence geometry and, subsequently, in harmonic analysis, leading to groundbreaking progress on some of the longest-standing conjectures in the fields. 

Polynomial partitioning relies on finding a low-degree variety containing the points we wish to count. This reveals hidden structure and allows computations. Indeed, in an incidence geometric problem, the points in question generally arise as intersections of certain algabraic objects (e.g., lines), and the study of the interaction of these objects with the variety leads to tighter bounds the lower the degree of the variety is (e.g., via B\'ezout's theorem). 

In stark contrast to this, our method relies on finding high-degree varieties containing our points. This allows us to take advantage of the (at first, hidden) grid-like structure of our point set in the absence of any incidence-geometric structure.

\subsection{Notation and conventions}
We will frequently use Vinogradov's asymptotic notation throughout this paper: we write $X \ll Y$ to mean that $X = O(Y)$, $X\gg Y$ to mean that $X = \Omega(Y)$ (i.e., $Y=O(X)$), and $X\asymp Y$ to mean that $X\ll Y$ and $Y\ll X$. We will write $O(X)$ to represent a quantity that is $\ll X$ and $\Omega(Y)$ to
represent a quantity that is $\gg Y$. We use $|S|$ to denote the cardinality of a set $S$,
$\|\cdot\|$ to denote the usual Euclidean norm on $\mathbf{R}^2$, $\mathbf{P}^n$ to denote
$n$-dimensional complex projective space, $\mathbf{P}_{\mathbf{Q}}^n$ to denote the
rational points of $\mathbf{P}^n$, $\tau$ to denote the divisor function, and
$\left( \frac{\cdot}{\cdot}\right)$ to denote the Legendre symbol. Finally, we will always
use $p$ to denote a prime number, and, for any natural number $n$, will write
$p^{a}\mid\mspace{-2mu}\mid n$ with $a\in\mathbf{N}$ to mean that $p^{a}$ divides $n$ but $p^{a+1}$
does not divide $n$ (so that $p^a$ is the highest power of $p$ dividing $n$).

\subsection*{Acknowledgments}
This work was initiated at the Hausdorff Research Institute for Mathematics (HIM) during the trimester program Harmonic Analysis and Analytic Number Theory; we are grateful to HIM for facilitating this collaboration. We also gratefully acknowledge the hospitality and support of the Institute for Advanced Study during the Special Year on Dynamics, Additive Number Theory and Algebraic Geometry, where a significant portion of this research was conducted. 

We thank Tony Carbery, Polona Durcik, Ben Green, Leo Hochfilzer, Karen Smith, and Terence Tao for
helpful discussions, J\'ozsef Solymosi for help with
references and informing us about the history of the study of integer distance sets, as well as Ben Green,  Alex Iosevich, and Noah Kravitz for useful comments on an earlier version of this paper. We would particularly like to thank Terence Tao for suggesting a strengthening of a prior version of our structure theorem.

RG was supported by the Association of Members of the Institute for Advanced Study (AMIAS) and by NSF grant DMS-2242871. SP was supported by the NSF Mathematical Sciences Postdoctoral Research Fellowship Program under Grant No. DMS-1903038 while most of this research was carried out.

\section{Setup and overview of the argument}\label{sec:setup}

We begin by proving the following lemma, which allows us to assume that the elements of an integer distance set are contained in a lattice and whose proof is a quantification of a well-known argument of Kemnitz~\cite{Kemnitz1988}.

\begin{lemma}\label{lem:lattice}
  Let $S\subset\mathbf{R}^2$ be an integer distance set containing the
  origin and at least one other point. Set $M\in\mathbf{N}$ to be the smallest distance
  from the origin to another (distinct) point in $S$. Then there exist a rotation
  $T\in\SO_2(\mathbf{R})$ and a
  squarefree $m\in \mathbf{N}$ such that
  \begin{equation*}
    TS\subset\left\{(x,y\sqrt{m}):x,y\in\frac{1}{2M}\mathbf{Z}\right\}.
  \end{equation*}
\end{lemma}
\begin{proof}
  Observe first that if $P_1$ and $P_2$ are any points in an integer distance set, then, since
  \begin{equation}\label{eq:dotprod}
    \|P_1\|^2+\|P_2\|^2-\|P_1-P_2\|^2=2P_1\cdot P_2
  \end{equation}
  by the law of cosines, the dot product $P_1\cdot P_2$ is in
  $\frac{1}{2}\mathbf{Z}$. There exists a rotation $T\in \SO_2(\mathbf{R})$ such that $(M,0)\in TS$, so if
  $(x,y)\in TS$, then $x$ is forced to lie in $\frac{1}{2M}\mathbf{Z}$. It therefore
  follows that all points in $TS$ that lie on the $x$-axis are contained in
  \begin{equation*}\label{eq:xaxis}
    \left\{(x,0):x\in\frac{1}{2M}\mathbf{Z}\right\}.
  \end{equation*}
  
  So, suppose that $(a/2M,b)$, with $a\in\mathbf{Z}$, is a point in $TS$ not on the
  $x$-axis. Then, since $(a/2M,b)\cdot (a/2M,b)$ is the square of the distance from
  $(a/2M,b)$ to the origin, we certainly have
  \begin{equation*}
    \frac{a^2}{4M^2}+b^2\in\mathbf{Z},
  \end{equation*}
  so that there exists an integer $k\in\mathbf{Z}$ such that
  \begin{equation*}
    b^2=\frac{4kM^2-a^2}{4M^2}.
  \end{equation*}
  It follows that there exist a squarefree $m\in\mathbf{N}$ and an integer $n\in\mathbf{Z}$ such
  that $b=\frac{n}{2M}\sqrt{m}$. Observe that this $m$ is the same for all points in $TS$ not
  on the $x$-axis. Indeed, if $$\left(\frac{a'}{2M},\frac{n'}{2M}\sqrt{m'}\right)\in TS$$ for $a',n',m'\in\mathbf{Z}$
  with $m'\geq 1$ squarefree, then, by \eqref{eq:dotprod}, the dot product 
  \begin{equation*}
    \left(\frac{a}{2M},\frac{n}{2M}\sqrt{m}\right)\cdot \left(\frac{a'}{2M},\frac{n'}{2M}\sqrt{m'}\right)=\frac{aa'}{4M^2}+\frac{nn'}{4M^2}\sqrt{mm'}
  \end{equation*}
  is a rational number, which forces $\sqrt{mm'}\in\mathbf{Q}$, and since $m$ and $m'$ are
  both squarefree, they have to be equal. We therefore conclude that $TS$ is contained in a set of the
  form
  \begin{equation*}
    \left\{(x,y\sqrt{m}):x,y\in\frac{1}{2M}\mathbf{Z}\right\},
  \end{equation*}
  as stated.
\end{proof}

\subsection{Encoding the points of $S$ as rational points on a surface} 
We will use basic definitions and results from algebraic geometry throughout this paper; a good reference for this is~\cite[Chapter I]{Hartshorne77}. Note that we may assume that $N$ is larger than any specified absolute constant, since the results are trivial for bounded $N$. In addition, observe that it
suffices to prove Theorem~\ref{thm:main} in the case that $S$ contains
the origin and at least one other point. Indeed, either $S$ contains at least two points
$P$ and $P'$ in $[-N,N]^2$, or the result is trivial; moreover, if there exists a line or circle $C\subset\mathbf{R}^2$ such that all but $O((\log{2N})^{O(1)})$ points of $(S-P)\cap[-2N,2N]^2$ are contained in $C$, then certainly all but $O((\log{N})^{O(1)})$ points of $S$ are contained in a single line or circle. Thus, it follows from Lemma~\ref{lem:lattice} that we may further assume,
without loss of generality, that there exists a squarefree $m\in\mathbf{N}$ such that
\[
  S\subset \left\{\left(x,y\sqrt{m}\right):x,y\in\frac{1}{2M}\mathbf{Z}\right\},
\]
where $M$ is the shortest distance from the origin to another point in $S$, which is necessarily at most $\sqrt{2}N$.

Let $k=k(N)$ be an integer parameter increasing slowly with $N$ (which we will,
eventually, take to be $\asymp \log\log{N}$). If $|S|\leq k+1$, then we are
done, since $N$ is large, by our initial assumption. Otherwise, pick $k$ distinct points $\tilde{P}_j=(a_j,b_j\sqrt{m})\in S$ (so that
$a_i,b_i\in\frac{1}{2M}\mathbf{Z}\subseteq \mathbf{Q}$) for $1\leq j\leq k$, none of which
is the origin, and set $P_j=(a_j,b_j)$ for each $1\leq j\leq k$. For any point $P=(a,b)\in\mathbf{Q}^2$, we define the polynomial $Q_{m,P}\in\mathbf{Q}[x,y,d]$ by
\begin{equation*}
 Q_{m,P}(x,y,d):=(x-a)^2+m(y-b)^2-d^2.
\end{equation*} 
Consider the affine variety
\begin{equation*}
  X_k:=\left\{(x,y,d_1,\dots,d_k):Q_{m,P_j}(x,y,d_j)=0\text{ for }j=1,\dots,k\right\}\subset\mathbf{C}^{k+2},
\end{equation*}
which is defined over $\mathbf{Q}$. Note that $\tilde{P}=(x,y\sqrt{m})\in S$ implies that
\begin{equation*}
  \left(x,y,\|\tilde{P}-\tilde{P}_1\|,\dots,\|\tilde{P}-\tilde{P}_k\|\right)\in X_k\cap \left(\frac{1}{2M}\mathbf{Z}\times\frac{1}{2M}\mathbf{Z}\times\mathbf{Z}^k\right).
\end{equation*}
Each such point corresponds to a point
\begin{equation*}
  \left[2Mx:2My:2M:2M\|\tilde{P}-\tilde{P}_1\|:\cdots:2M\|\tilde{P}-\tilde{P}_k\|\right]
\end{equation*}
with all components integers in the projective closure
$\overline{X_k}\subset\mathbf{P}^{k+2}$ of $X_k$. Thus, the points in $S$ correspond to
a certain subset of rational points in $\overline{X_k}$ of height at most $8N^2$.  In
Section~\ref{sec:ag}, we will prove that $\overline{X_k}$ is an irreducible surface of degree $2^k$.

\subsection{Covering rational points of small height by few low degree curves}\label{subsec:2.4}
Let $V\subset\mathbf{P}^n$ be a projective variety defined over $\mathbf{Q}$. For $N\geq
1$, we set
\begin{equation*}
    V(\mathbf{Q},N) := \left\{z\in V\cap\mathbf{P}^n_{\mathbf{Q}}: H(z)\leq N\right\},
\end{equation*}
where $H(z) := \max\left\{|z_0|,\dots,|z_n|\right\}$ for $z=[z_0:\dots:z_n]$ with
$z_0,\dots,z_n\in\mathbf{Z}$ coprime denotes the \emph{height} of a rational point $z\in V$.

It is known, originally thanks to work of Heath-Brown~\cite{Heath-Brown2002}, that almost
all rational points of small height on an irreducible projective surface defined over
$\mathbf{Q}$ lie on a small number of low degree
curves. Salberger~\cite{Salberger2023} refined this result in the course of his proof of
the Dimension Growth Conjecture, and then Salberger's result was further refined by Walsh~\cite{Walsh2015} to remove all factors of $\log{N}$ and $N^{c/\log(1+\log{N})}$ appearing in the bounds. We will apply a subsequent refinement of Walsh's result
(Theorem~\ref{thm:surfacecount} below) that follows from work of Castryk, Cluckers,
Dittmann, and Nguyen~\cite{CastryckCluckersDittmannNguyen2020}, which makes explicit the dependence of various implied constants on the degree of the surface and the dimension of the ambient projective space. This will allow us to deduce, after an
appropriate projection, that all points in $S$ lie on the union of $\ll
e^{O(k)}N^{3/2^{k/2}}$ irreducible affine curves of degree $\ll e^{O(k)}N^{3/2^{k/2}}$. 

\begin{theorem}\label{thm:surfacecount}
  For any irreducible surface $V\subset\mathbf{P}^n$ of degree $d$ defined over
  $\mathbf{Q}$, there exists a homogeneous polynomial $g\in\mathbf{Z}[x_0,\dots,x_n]$ of
  degree at most
    \begin{equation}\label{eq:surfacecount}
        \ll n^{175n^2/\sqrt{d}}d^{7/2+3n^2/\sqrt{d}} N^{3/2\sqrt{d}},
    \end{equation}
    that vanishes at all rational points of $V$ of height at most $N$ but does not vanish
    on the whole of $V$.
\end{theorem}

\begin{proof}
    By tracing through the proof of Lemma~5.1 of~\cite{CastryckCluckersDittmannNguyen2020}, the implied constant $c_n$ in the lemma statement satisfies $c_n\ll n^{100n^2}$ and the constant $B_2$ in the proof satisfies $B_2\ll n^{10n}d^{2n}$. Thus, by this lemma and the definition of $\Gamma$ in its proof, there exists a linear map $L:\mathbf{P}^n\to\mathbf{P}^3$ defined over the rationals such that $L(V)$ is an irreducible surface of degree $d$ and such that, if $z\in V$ is a rational point, then
    \begin{equation*}
        H(L(z))\leq c_nd^{2(n-3)^2}H(z)\ll n^{115n^2}d^{2n^2} H(z).
    \end{equation*}
     Applying Theorem~3.1.1 of~\cite{CastryckCluckersDittmannNguyen2020} to $L(V)$ produces a homogeneous polynomial $h\in\mathbf{Z}[x_0,x_1,x_2,x_3]$ of degree
    \begin{equation*}
        \ll (c_nd^{2(n-3)^2}N)^{3/2\sqrt{d}}d^{7/2}\ll n^{175n^2/\sqrt{d}}d^{7/2+3n^2/\sqrt{d}} N^{3/2\sqrt{d}}
    \end{equation*}
    (using that the quantity $b(f)/\|f\|^{1/2d^{3/2}}$ appearing in the theorem statement in~\cite{CastryckCluckersDittmannNguyen2020} is bounded) that vanishes on all rational points of $L(V)$ of height at most $c_nd^{2(n-3)^2}N$ but does not vanish on all of $L(V)$. This means that the homogeneous polynomial $g:=h\circ L\in\mathbf{Q}[x_0,\dots,x_n]$ has degree at most $\deg{h}$, vanishes on every rational point of $V$ of height at most $N$, and does not vanish on all of $V$.
\end{proof}

\subsection{Counting rational points of small height on curves} As mentioned above, by applying Theorem~\ref{thm:surfacecount} to $\overline{X_k}$ and then an
  appropriate projection (see Lemma~\ref{lem:projection}), we deduce that $S$ is covered by $\ll
e^{O(k)}N^{3/2^{k/2}}$ irreducible affine curves of degree $\ll
e^{O(k)}N^{3/2^{k/2}}$. Taking $k\asymp\log\log N$ reveals that $S$ is covered by
$\ll(\log N)^{O(1)}$ irreducible affine curves of degree $\ll(\log N)^{O(1)}$. Therefore,
to prove Theorem~\ref{thm:main}, it suffices to bound, for each irreducible curve $C$ that either (1) is not a line or a circle or (2) is a line or circle for which $|S\setminus C|\gg(\log\log{N})^2$, the maximum possible size of an integer distance set contained
in $C\cap [-N,N]^2$ by $\ll(\log N)^{O(1)}$. Indeed, this immediately gives that all but at most $\ll(\log{N})^{O(1)}$ points of $S$ are contained in a single line or circle.

We show in Section~\ref{sec:end} that either $C$ contains few points in $S$, or $C':=\{(x,y):(x,y\sqrt{m})\in C\}$ is defined over $\mathbf{Q}$. In the latter case, we proceed as before, selecting $k$ distinct, well-chosen points $\tilde{P}_1',\dots,\tilde{P}_k'\in S$ depending on $C$, none of which is the origin, with $\tilde{P}_j'=(a_j',b_j'\sqrt{m})$ and $P_j'=(a_j',b_j')\in\mathbf{Q}^2$ for each $1\leq j\leq k$, and consider the affine variety
\begin{equation*}\label{eq:curvek}
  \mathcal{C}_k:=\left\{(x,y,d_1,\dots,d_k):Q_{m,P_j'}(x,y,d_j)=0\text{ for all }j=1,\dots,k\right\}\cap(C'\times\mathbf{C}^k)\subset\mathbf{C}^{k+2}.
\end{equation*}
Then, the points in $S\cap C$ correspond to a certain subset of rational points in the projective closure $\overline{\mathcal{C}_k}\subset\mathbf{P}^{k+2}$ of height at most $8N^2$. In Section~\ref{sec:ag2}, we will prove that $\overline{\mathcal{C}_k}$ is an irreducible curve of degree between $2^k$ and $2^{k}\deg{C}$, provided that $\tilde{P}_1',\dots,\tilde{P}_k'\in S$ are chosen appropriately. As was mentioned in the introduction, proving the irreducibility of $\overline{\mathcal{C}_k}$ is the most difficult part of the argument. In contrast to the proof of the irreducibility of $\overline{X_k}$, which is quite straightforward, the proof of the irreducibility of $\overline{\mathcal{C}_k}$ requires the normalization theorem for irreducible algebraic curves and the classification of compact Riemann surfaces combined with a detailed analysis of the geometry of the intersection of $C'$ with a certain family of lines depending on the choice of the points $\tilde{P}_1',\dots,\tilde{P}_k'$. We emphasize that these points will be chosen on $C$ when $C$ is not a line or a circle, and outside of $C$ otherwise.

We can then apply a refinement of the Bombieri--Pila~\cite{BombieriPila1989} bound due to
Castryk, Cluckers, Dittmann, and Nguyen~\cite{CastryckCluckersDittmannNguyen2020}, who
proved an upper bound on $|V(\mathbf{Q},N)|$, when $V\subset\mathbf{P}^n$ is an
irreducible algebraic curve defined over $\mathbf{Q}$, with optimal dependence on $N$ and
with an explicit polynomial dependence on the degree of $V$. By tracing through the proof
of \cite[Theorem~2]{CastryckCluckersDittmannNguyen2020} (and the results used to deduce it
from the $n=2$ case \cite[Proposition~4.3.2 and
Lemma~5.1]{CastryckCluckersDittmannNguyen2020}), one can see that the dependence of their
implied constant on $n$ is $\ll n^{200n^2/d}d^{4n^2/d}$. We record this in the following
theorem.
\begin{theorem}[Castryk, Cluckers, Dittmann, Nguyen~\cite{CastryckCluckersDittmannNguyen2020}]\label{thm:curvecount}
    For any irreducible curve $V\subset\mathbf{P}^n$ of degree $d$ defined over $\mathbf{Q}$, we have
    \begin{equation}\label{eq:curvecount}
        |V(\mathbf{Q},N)|\ll n^{200n^2/d}d^{4(1+n^2/d)} N^{2/d}.
    \end{equation}
\end{theorem}

Thus, as long as $C$ either is not a line or circle or is a line or circle for which $|S\setminus C|\gg(\log\log{N})^2$, we can use Theorem~\ref{eq:curvecount} to bound the number
of points in $S\cap C$ by $\ll e^{O(k)}N^{2/2^{k}}$, which, for $k\asymp
\log\log N$, is $\ll(\log N)^{O(1)}$. Putting everything together yields
Theorem~\ref{thm:main}. We will carry out the argument just outlined in Sections~\ref{sec:ag},~\ref{sec:ag2}, and~\ref{sec:end}, and prove
Propositions~\ref{prop:lines} and~\ref{prop:main} in Section~\ref{sec:circles}.

\begin{remark}
    Some polynomial dependence on $d$ is necessary in the bounds in Theorems~~\ref{thm:surfacecount} and \ref{thm:curvecount} (see~\cite[Proposition 5]{CastryckCluckersDittmannNguyen2020}). 
    As a consequence, no bound better than $|S\setminus C|\ll(\log{N})^{O(1)}$ in Theorem~\ref{thm:main} seems attainable via our methods. The determinant method produces bounds that apply to any irreducible surface (respectively, curve) in $\mathbf{P}^n$ of degree $d$, but significantly better bounds could hold for individual surfaces (respectively, curves) in this collection. Thus, a further analysis of the specific varieties arising in our argument would be needed to improve Theorem~\ref{thm:main}.
\end{remark}

\section{Degree and irreducibility of $\overline{X_k}$}\label{sec:ag}

The goal of this section is to prove the following result:

\begin{lemma}\label{lem:generalX}
    Let $k\in\mathbf{N}$, $m\in\mathbf{N}$ be squarefree, and $P_1=(a_1,b_1),\dots,P_k=(a_k,b_k)\in\mathbf{Q}^2$ be any $k$ distinct points, none of which is the origin. Define the affine variety $X_k\subset\mathbf{C}^{k+2}$  by
    \begin{equation}\label{eq:Xk}
        X_k:=\left\{(x,y,d_1,\dots,d_k)\in\mathbf{C}^{k+2}:Q_{m,P_j}(x,y,d_j)=0\text{ for all }j=1,\dots,k\right\}.
    \end{equation}
    Then $\overline{X_k}$ is an irreducible surface in $\mathbf{P}^{k+2}$ of degree $2^k$ defined over $\mathbf{Q}$.
\end{lemma}
 For the remainder of this section, $X_k$ will refer to $X_k$ as in the above lemma. There
are several paths to showing that $\overline{X_k}$ is a surface of degree $2^k$; we will take
one of the longer routes, since some of the intermediate lemmas we prove along the way
will be useful later in this section and in Section~\ref{sec:ag2}.

\subsection{An explicit description of $\overline{X_k}$}

We begin by explicitly describing $\overline{X_k}$.
\begin{lemma}\label{lem:Xbar}
  Let $\overline{X_k}$ be as in Lemma~\ref{lem:generalX}. We have
  \begin{equation*}
  \overline{X_k} = \left\{[x:y:z:d_1:\dots:d_k]:\overline{Q_{m,P_j}}(x,y,z,d_j)=0\text{ for }j=1,\dots,k\right\},
\end{equation*}
where
\begin{equation*}
 \overline{Q_{m,P_j}}(x,y,z,d_j)=(x-a_jz)^2+m(y-b_jz)^2-d_j^2
\end{equation*}
for each $j=1,\dots,k$.
\end{lemma}

Let $\mathcal{M}$ denote the set of monomials in the variables $x,y,d_1,\dots,d_k$. Any
polynomial $Q\in\mathbf{C}[x,y,d_1,\dots,d_k]$ can be written as
\begin{equation*}
  Q=\sum_{\mathbf{m}\in \mathcal{M}}a_{\mathbf{m}}\mathbf{m},
\end{equation*}
where $a_{\mathbf{m}}=0$ for all but finitely many $\mathbf{m}\in \mathcal{M}$. We say that a
monomial $\mathbf{m}$ \textit{appears} in $Q$ if $a_{\mathbf{m}}\neq 0$ in the above expression, and denote the homogenization of $Q$ with respect to $z$ by $\overline{Q}:=z^{\deg{Q}}Q(x/z,y/z,d_1/z,\dots,d_k/z)$, where $\deg{Q}$ is the maximal total degree of any monomial appearing in $Q$.

 To prove Lemma~\ref{lem:Xbar}, it suffices, by standard results (e.g., combine
 Theorems~4 and~8 of~\cite[Chapter 8 \S 4]{CoxLittleOShea07}), to show that
 $Q_{m,P_1},\dots,Q_{m,P_k}$ is a Gr\"obner basis\footnote{See Definition~5 in
   \cite[Chaper 2 \S 5]{CoxLittleOShea07} for the definition of a Gr\"obner basis.} for $I:=\langle Q_{m,P_1},\dots,Q_{m,P_k}\rangle$ with respect to some graded monomial order. Recall that a \textit{monomial order} is a total order $\succ$ on $\mathcal{M}$ satisfying
\begin{enumerate}
\item $\mathbf{m}\succ 1$ for all $1\neq \mathbf{m}\in \mathcal{M}$ and
\item $\mathbf{m}\succ\mathbf{m}'$ implies that $\mathbf{m}\mathbf{m}''\succ\mathbf{m}'\mathbf{m}''$ for all $\mathbf{m}''\in \mathcal{M}$.
\end{enumerate}
A monomial order $\succ$ on $\mathcal{M}$ is said to be \textit{graded} if
\begin{equation*}
  d_1^{c_1}\cdots d_k^{c_k}x^{c_{k+1}}y^{c_{k+2}}\succ d_1^{c'_1}\cdots d_k^{c'_k}x^{c_{k+1}'}y^{c_{k+2}'}
\end{equation*}
whenever $c_1+\dots+c_{k+2}>c'_1+\dots+c'_{k+2}$. Once a monomial order $\succ$ has been fixed, one can write any $Q\in\mathbf{C}[x,y,d_1,\dots,d_k]$ as
\begin{equation*}
  Q=a_1\mathbf{m}_1+\dots+a_n\mathbf{m}_n,
\end{equation*}
where $\mathbf{m}_j\in \mathcal{M}$ and $a_j\neq 0$ for all $j=1,\dots,n$, and where
$\mathbf{m}_1\succ\dots\succ\mathbf{m}_n$. Then the \textit{leading term} of $Q$ is
defined to be $$\lt{Q}:= a_1\mathbf{m}_1,$$ and the \textit{leading monomial} of $Q$ is
defined to be $$\lm{Q}:=\mathbf{m}_1.$$ 
Finally, recall that the multivariable polynomial division algorithm (see, for
example,~\cite[Chapter 2 \S 3]{CoxLittleOShea07}) says that \textit{the remainder of $Q$ under division by a set of polynomials $G=\{g_1,\dots,g_n\}\subset\mathbf{C}[x,y,d_1,\dots,d_k]$ is zero} if and only if we can write
\begin{equation*}
  Q=h_1g_1+\dots+h_ng_n,
\end{equation*}
where each $h_j\in\mathbf{C}[x,y,d_1,\dots,d_k]$ and $\lm{Q}\succeq \lm{h_jg_j}$ whenever $h_jg_j\neq 0$.

To prove Lemma~\ref{lem:Xbar}, we will use the graded lexicographic (grlex) order $\succ$ with $d_1\succ\dots\succ d_k\succ x\succ y$, which is defined by declaring that $d_1^{c_1}\cdots d_k^{c_k}x^{c_{k+1}}y^{c_{k+2}}\succ d_1^{c'_1}\cdots d_k^{c'_k}x^{c_{k+1}'}y^{c_{k+2}'}$ if and only if one of the following two conditions is met:
\begin{enumerate}
\item $c_1+\dots+c_{k+2}>c'_1+\dots+c'_{k+2}$ or
\item $c_1+\dots+c_{k+2}=c'_1+\dots+c'_{k+2}$ and $c_{j}>c_{j}'$ if $j$ is the smallest index for which $c_j\neq c'_j$.
\end{enumerate}
We will also require Buchberger's criterion (see~\cite[Chapter 2]{CoxLittleOShea07}), which gives a convenient way to check whether a generating set of an ideal is a Gr\"obner basis.
\begin{lemma}[Buchberger's criterion]
  Let $J$ be an ideal of $\mathbf{C}[x,y,d_1,\dots,d_k]$ and $G=\{g_1,\dots,g_n\}$ be a basis for $J$. Then $G$ is a Gr\"obner basis with respect to a monomial order $\succ$ if and only if each $S$-polynomial
  \begin{equation*}
    S(g_i,g_j) := \frac{\lcm(\lm{g_i},\lm{g_j})}{\lt{g_i}}g_i-\frac{\lcm(\lm{g_i},\lm{g_j})}{\lt{g_j}}g_j
  \end{equation*}
  has remainder zero on division by $G$.
\end{lemma}

\begin{lemma}\label{lem:gb}
With respect to the grlex ordering $\succ$ with $d_1\succ\dots\succ d_k\succ x\succ y$ on $\mathcal{M}$, $Q_{m,P_1},\dots,Q_{m,P_k}$ is a Gr\"obner basis for $I=\langle Q_{m,P_1},\dots,Q_{m,P_k}\rangle$.
\end{lemma}
\begin{proof}
By Buchberger's criterion, it suffices to check that each $S$-polynomial $S(Q_{m,P_i},Q_{m,P_j})$
  with $i\neq j$ has remainder zero on division by 
  \begin{equation*}
      G:=\left\{Q_{m,P_1},\dots,Q_{m,P_k}\right\}.
  \end{equation*}
    So let $i<j$, and
  note that, since $\lm{Q_{m,P_i}}=\lt{Q_{m,P_i}}=d_i^2$ and $\lm{Q_{m,P_j}}=\lt{Q_{m,P_j}}=d_j^2$, we have
  \begin{align*}
    S(Q_{m,P_i},Q_{m,P_j}) &= d_j^2 Q_{m,P_i} - d_i^2 Q_{m,P_j} \\
               &= d_j^2[(x-a_i)^2+m(y-b_i)^2]-d_i^2[(x-a_j)^2+m(y-b_j)^2] \\
               &= -Q_{m,P_j}[(x-a_i)^2+m(y-b_i)^2]+Q_{m,P_i}[(x-a_j)^2+m(y-b_j)^2].
  \end{align*}
  From the second line above, we see that $\lm{S(Q_{m,P_i},Q_{m,P_j})}=d_i^2x^2$, so since
  \[\lm{-Q_{m,P_j}[(x-a_i)^2+m(y-b_i)^2]}=d_j^2x^2\preceq \lm{S(Q_{m,P_i},Q_{m,P_j})}\] and
  \[\lm{Q_{m,P_i}[(x-a_j)^2+m(y-b_j)^2]}=d_i^2x^2\preceq \lm{S(Q_{m,P_i},Q_{m,P_j})},\] we certainly have that $S(Q_{m,P_i},Q_{m,P_j})$ has
  remainder zero on division by $G$.
\end{proof}

By the discussion above, Lemma~\ref{lem:Xbar} now follows.

\subsection{Dimension and degree of $\overline{X_k}$}

Next, we will confirm that $\overline{X_k}$ is a surface of degree $2^k$. There are a few
ways to see that $\overline{X_k}$ is a projective surface, but perhaps the fastest is to
note that $\dim{\overline{X_k}}=\dim{X_k}$, and to use the general fact (see~\cite[Chapter
9 \S 3]{CoxLittleOShea07}) that if $J\subset\mathbf{C}[x,y,d_1,\dots,d_k]$ is an ideal and $G=\{g_1,\dots,g_\ell\}$ is a Gr\"obner basis for $J$ with respect to a graded monomial order, then $\dim{V(J)}$ equals the maximal size of a subset $A$ of the variables $\{x,y,d_1,\dots,d_k\}$ such that, for all $j=1,\dots,\ell$, the leading term of $g_j$ contains a variable not in $A$. The following lemma is then an immediate consequence of Lemma~\ref{lem:gb} (by taking $A=\{x,y\}$).

\begin{lemma}\label{lem:surface}
  Let $\overline{X_k}$ be as in Lemma~\ref{lem:generalX}. Then $\dim\overline{X_k}=2$ (i.e., $\overline{X_k}$ is a surface).
\end{lemma}

Now, we will show that $\overline{X_k}$ has degree $2^k$. This will follow easily from two
standard facts, which we record here for later use. Both are immediate consequences
of Theorem 7.7 from~\cite[Chapter I]{Hartshorne77}.

\begin{lemma}[B\'ezout's inequality]\label{lem:Bezout}
Let $V=V(g_1,\dots,g_\ell)\subset\mathbf{P}^n$ be a projective variety. We have
\begin{equation*}
  \deg{V}\leq\prod_{j=1}^\ell\deg{g_j}.
\end{equation*}
\end{lemma}

\begin{lemma}\label{lem:deglow}
Let $V\subset\mathbf{P}^n$ be a projective variety of dimension $d$. If
$H\subset\mathbf{P}^n$ is any linear subvariety of codimension $d$ such that $|V\cap H|$
is finite, then $\deg{V}$ is at
least $|V\cap H|$.
\end{lemma} 

Now we can finish this subsection.
\begin{lemma}\label{lem:generalXdeg}
Let $\overline{X_k}$ be as in Lemma~\ref{lem:generalX}. Then $\deg{\overline{X_k}}=2^k$.  
\end{lemma}
\begin{proof}
  By Lemmas~\ref{lem:Xbar} and~\ref{lem:Bezout}, we certainly have $\deg{\overline{X_k}}\leq
  2^k$. For the lower bound, consider the linear variety $H$ of codimension $2$ defined by
  $H=V(x,y)$. Then
  \begin{equation*}
    \overline{X_k}\cap H=\left\{\left[0:0:1:\pm\sqrt{a_1^2+mb_1^2}:\dots:\pm\sqrt{a_k^2+mb_k^2}\right]\right\},
  \end{equation*}
  which has $2^k$ elements, since we stipulated that none of
  $(a_1,b_1),\dots,(a_k,b_k)$ are the origin. Thus, by Lemma~\ref{lem:deglow}, $\deg{\overline{X_k}}\geq
  2^k$ as well, so that $\deg{\overline{X_k}}=2^k$.
\end{proof}

\subsection{Irreducibility of $\overline{X_k}$}\label{sec:ag.irr}

Finally, we will show that $\overline{X_k}$ is irreducible.
\begin{lemma}\label{lem:generalXirred}
    Let $\overline{X_k}$ be as in Lemma~\ref{lem:generalX}. Then $\overline{X_k}$ is irreducible.
\end{lemma}
Note that it suffices to check that $X_k$ is an irreducible affine variety. We will begin this section by computing the set of singular points of $X_k$.

\begin{lemma}\label{lem:singularpts}
  Let $X_k$ be as in Lemma~\ref{lem:generalX}, and denote the set of singular points of $X_k$ by $X'_k$. We have
  \begin{equation*}
    X'_k=\left\{(x,y,d_1,\dots,d_k)\in X_k:(x,y)=P_1,\dots,P_k\right\}.
  \end{equation*}
  Thus, $X'_k$ consists of $k2^{k-1}$ distinct points.
\end{lemma}
\begin{proof} 
  Observe that $Q_{m,P_j}(x,y,0)=(x-a_j)^2+m(y-b_j)^2$ factors as
  \begin{equation*}
    Q_{m,P_j}(x,y,0)=\left(x+i\sqrt{m}y-(a_j+i\sqrt{m}b_j)\right)\left(x-i\sqrt{m}y-(a_j-i\sqrt{m}b_j)\right).
  \end{equation*}
  Thus, $Q_{m,P_j}(x,y,0)=0$ if and only if the pair $(x,y)\in\mathbf{C}^2$ lies on at least one of the lines
  \begin{equation*}
    \ell_j : x+i\sqrt{m}y=a_j+i\sqrt{m}b_j
  \end{equation*}
  or
  \begin{equation*}
   \ell_{j}' : x-i\sqrt{m}y=a_{j}-i\sqrt{m}b_{j},
  \end{equation*}
  which intersect at the point $(x,y)=P_j$. Since $P_1,\dots,P_k$ are all distinct, we must have $\ell_j\cap\ell_{j'}=\ell_{j}'\cap\ell_{j'}'=\emptyset$ whenever $j\neq j'$. A short calculation shows that $\ell_j$ and $\ell_{j'}'$ intersect at the point $\left(\frac{a_j+a_{j'}}{2}+i\sqrt{m}\frac{b_j-b_{j'}}{2},\frac{b_j+b_{j'}}{2}-i\frac{a_j-a_{j'}}{2\sqrt{m}}\right)$ when $j\neq j'$.
  
The Jacobian matrix of $X_k$ at a point $P=(x,y,d_1,\dots,d_k)$ is the $k\times (k+2)$ matrix
\begin{equation*}
  \begin{pmatrix}
    2(x-a_1) & 2m(y-b_1) & -2d_1 & 0 & \dots & 0 \\
    2(x-a_2) & 2m(y-b_2) & 0 & -2d_2 & \dots & 0 \\
    \vdots & \vdots & \vdots & \vdots & \ddots & \vdots \\
    2(x-a_k) & 2m(y-b_k) & 0 & \dots & 0 & -2d_k
  \end{pmatrix}.
\end{equation*}
If $P=(x,y,d_1,\dots,d_k)\in X_k$ with $(x,y)=P_j$ for some $i=1,\dots,k$, then $P$ is
clearly a singular point of $X_k$, as the $j$-th row of the Jacobian matrix is zero and
thus the rank cannot be $k$. Conversely, $P=(x,y,d_1,\dots,d_k)$ can only possibly be a
singular point of $X_k$ if $d_j=0$ for at least one $j=1,\dots,k$. If $d_j=0$ for exactly
one index $j$ and $(x,y)\neq P_j$, then $(x-a_j,y-b_j)\neq (0,0)$ and the Jacobian matrix
still has rank $k$. Since $P_1,\dots,P_k$ are all distinct, it follows that $\ell_j\cap\ell_{j'}\cap\ell_{j''}'=\emptyset$ and $\ell_j\cap\ell_{j'}'\cap\ell_{j''}'=\emptyset$ whenever $j,j'$ and $j''$ are pairwise distinct indices. Thus, the only remaining possibility is that $d_j=d_{j'}=0$ for exactly two distinct indices $j$ and $j'$, in which case $(x,y)$ is forced to be the intersection point of either $\ell_j$ and $\ell_{j'}'$ or of $\ell_{j'}$ and $\ell_j'$. Without loss of generality, we may assume that we are in the first case, so that
\begin{equation*}
  \begin{pmatrix}
    x-a_j & y-b_j \\
    x-a_{j'} & y-b_{j'}
  \end{pmatrix} =   \begin{pmatrix}
    \frac{-a_j+a_{j'}}{2}+i\sqrt{m}\frac{b_j-b_{j'}}{2} & \frac{-b_j+b_{j'}}{2}-i\frac{a_j-a_{j'}}{2\sqrt{m}} \\
    \frac{a_j-a_{j'}}{2}+i\sqrt{m}\frac{b_j-b_{j'}}{2} & \frac{b_j-b_{j'}}{2}-i\frac{a_j-a_{j'}}{2\sqrt{m}}
  \end{pmatrix}.
\end{equation*}
This matrix has determinant
$\frac{i}{2}\left(\sqrt{m}(b_j-b_{j'})^2+\frac{(a_j-a_{j'})^2}{\sqrt{m}}\right)\neq 0$,
and so is nonsingular. Thus, in this case too, the Jacobian matrix still has rank $k$. We conclude that $(x,y)=P_j$ for some $1\leq j\leq k$ is necessary and sufficient for $P=(x,y,d_1,\dots,d_k)$ to be a singular point of $X_k$.
\end{proof}

Set $Y_k:= X_k\setminus X'_k$ to be the set of nonsingular points of $X_k$. Since $Y_k$ is
open and dense in $X_k$, to show that $X_k$ is irreducible, it suffices to show that $Y_k$
is irreducible. Further, since $Y_k$ is a smooth quasiaffine variety (as it is a smooth, open subvariety of $X_k$), it is irreducible if
it is connected (see, for example, Remark 7.9.1 of~\cite[Chapter III]{Hartshorne77}). Set
\begin{equation*}
  Y'_k:=\left\{(x,y,d_1,\dots,d_k)\in Y_k:Q_{m,P_j}(x,y,0)=0\text{ for some }j=1,\dots,k\right\},
\end{equation*}
so that, in the notation of the proof of Lemma~\ref{lem:singularpts},
\begin{equation}\label{eq:YkY'k}
 Y_k\setminus Y'_k=X_k\setminus\bigcup_{j=1}^k\pi^{-1}(\ell_j\cup\ell'_j),
\end{equation}
where $\pi:\mathbf{C}^{k+2}\to\mathbf{C}^2$ is the standard projection map onto the first
two coordinates. Note that $\dim{Y'_k}\leq 1$ by, for example, Theorem~1.22 of~\cite{Shafarevich2013}, since neither of $l_j(x,y):=(x-a_j)+i\sqrt{m}(y-b_j)$ nor $\tilde{l}_j:=(x-a_j)-i\sqrt{m}(y-b_j)$ vanish on the whole of $X_k$.
Indeed, we proved in Lemma~\ref{lem:gb} that $G=\{Q_{m,P_1},\dots,Q_{m,P_k}\}$ is a Gr\"obner basis for $I=\langle Q_{m,P_1},\dots,Q_{m,P_k}\rangle$ with respect to the grlex ordering $\succ$ with $d_1\succ\dots\succ d_k\succ x\succ y$. If $l_j$ were to vanish on $X_k$, then we would have $l_j\in\sqrt{I}$, so that there would exist an $n\in\mathbf{N}$ such that $l_j^n\in I$. But, since $Q_{m,P_1},\dots,Q_{m,P_k}$ is a Gr\"obner basis for $I$, $\ell_j^n\in I$ if and only if $\ell_j^n$ has remainder zero on division by $G$. Note that if
    \begin{equation*}
        l_j^n = g_1Q_{m,P_1}+\dots+g_kQ_{m,P_k}
    \end{equation*}
    for some $g_1,\dots,g_k\in\mathbf{C}[x,y,d_1,\dots,d_k]$, then we must, necessarily,
    have $\lt{g_{j'}Q_{m,P_{j'}}}\succ \lt{l_j^n}$ for some $1\leq j'\leq k$ (in particular, for any $j'$ for which
    $\deg l_j^n=\deg g_{j'}Q_{m,P_{j'}}$), since all monomials appearing in $l_j^n$ are divisible by
    only the variables $x$ and $y$, while if $g_{j'}Q_{m,P_{j'}}\neq 0$, then
    $\lt{g_{j'}Q_{m,P_{j'}}}$ is divisible by $d_{j'}^2$, and $d_{j'}\succ x\succ y$ for all
    $j'=1,\dots,k$. It follows that $l_j^n$ cannot possibly have remainder zero on division
    by $G$, so $l_j$ does not vanish on the whole of $X_k$. An identical argument shows that $\tilde{l}_j$ also does not vanish on the whole of $X_k$.

It follows that $Y_k\setminus Y'_k$ must be dense 
in $Y_k$, since all irreducible components of $X_k$ are two-dimensional (see, for example, Proposition~7.1 of~\cite[Chapter I]{Hartshorne77}). The proof of Lemma~\ref{lem:generalXirred} thus boils down to showing
that $Y_k\setminus Y'_k$ is connected. As any set that is connected in the standard (Euclidean) topology is Zariski-connected, it thus suffices to show that $Y_k\setminus Y'_k$ is connected in the
standard topology on $\mathbf{C}^{k+2}$, which we will work in for the remainder of this
section. We will next show that $Y_k\setminus Y'_k$ is, in fact,
path-connected. It will then follow that $Y_k$ is connected, and thus Zariski-connected,
as desired.

\begin{lemma}\label{lem:YnoY'}
  The set $Y_k\setminus Y'_k$ is path-connected.
\end{lemma}
\begin{proof}
    Set
    \begin{equation*}
        W:=\mathbf{C}^2\setminus\bigcup_{j=1}^k(\ell_j\cup\ell_j').
    \end{equation*}
    First, note that $W$ itself is path connected. Indeed, $W$ is homeomorphic to
    $\mathbf{R}^4$ minus $2k$ planes, and $\mathbf{R}^4$ minus any finite number of planes
    $H_1,\dots,H_m$ can be seen to be path connected by taking any two points $P,P'\in
    \mathbf{R}^4\setminus\bigcup_{j=1}^m H_j$, noting that there must exist a plane
    $H\subset\mathbf{R}^2$ containing both $P$ and $P'$ such that $H\cap\bigcup_{j=1}^m
    H_j$ consists of finitely many points, and then using that a plane minus any finite
    number of points is path connected. Observe that $\pi|_{Y_k\setminus Y'_k}: Y_k\setminus Y'_k\to W$ is a $2^k$-fold covering map.
    
    It therefore suffices to show that there exists $z_0\in Y_k\setminus Y'_k$ such that all points in the fiber $F:=(\pi|_{Y_k\setminus Y'_k})^{-1}(\pi(z_0))$ are in the same path-component in $Y_k\setminus Y'_k$. Indeed, for any two points $z,z'\in Y_k\setminus Y'_k$, there exist paths $\gamma_z$ and $\gamma_{z'}$ in $W$ from $\pi(z)$ to $\pi(z_0)$ and from $\pi(z_0)$ to $\pi(z')$, respectively. By~\eqref{eq:YkY'k}, these paths lift to paths $\widetilde{\gamma}_{z}$ and $\widetilde{\gamma}_{z'}$ in $Y_k\setminus Y'_k$ from $z$ to an element $z''$ of $F$ and from an element $z'''$ of $F$ to $z'$, respectively. If there exists a path $\widetilde{\gamma}$ from $z''$ to $z'''$, then concatenating $\widetilde{\gamma}_{z}$, $\widetilde{\gamma}$, and $\widetilde{\gamma}_{z'}$ shows that $z$ and $z'$ lie in the same path-component.

    We will select $z_0$ to be the explicit point
    \begin{equation*}
    z_0:=\left(0,0,\sqrt{a_1^2+mb_1^2},\dots,\sqrt{a_k^2+mb_k^2}\right)\in Y_k\setminus Y'_k,
    \end{equation*}
    so that $F$ equals
    \begin{equation*}
        \left\{\left(0,0,(-1)^{\epsilon_1}\sqrt{a_1^2+mb_1^2},\dots,(-1)^{\epsilon_k}\sqrt{a_k^2+mb_k^2}\right):\epsilon_1,\dots,\epsilon_k\in\{0,1\}\right\}.
    \end{equation*}
    Consider the (complex) line
    $H:=\left\{\left(z,\frac{z}{i\sqrt{m}}\right):z\in\mathbf{C}\right\}\subset\mathbf{C}^2$,
    and note that, since $H\cap\ell_j'=\emptyset$ for each $j=1,\dots,k$, we have
    \begin{equation*}
        H\cap W = H\setminus\bigcup_{j=1}^k(H\cap\ell_j),
    \end{equation*}
    where the intersection $H\cap\ell_j$ consists of the single point
    $\left(\frac{a_j+i\sqrt{m}b_j}{2},\frac{a_j+i\sqrt{m}b_j}{2i\sqrt{m}}\right)$ for each
    $j=1,\dots,k$. Moreover, since for $1\leq j<j'\leq k$ we have
    \begin{equation*}
    H\cap\ell_j=\left\{\left(\frac{a_j+i\sqrt{m}b_j}{2},\frac{a_j+i\sqrt{m}b_j}{2i\sqrt{m}}\right)\right\}\neq\left\{\left(\frac{a_{j'}+i\sqrt{m}b_{j'}}{2},\frac{a_{j'}+i\sqrt{m}b_{j'}}{2i\sqrt{m}}\right)\right\}= H\cap\ell_{j'},  
    \end{equation*}
    the set $\bigcup_{j=1}^k H\cap\ell_j$ consists of $k$ distinct elements.

    Since $H\cap W$ is homeomorphic to a (real) plane minus $k$ points, for each
    $A\subset\{1,\dots,k\}$ there certainly exists a path $\gamma$ contained in $H\cap W$
    starting and ending at $(0,0)$ with winding number $1$ around each of the points in
    $\bigcup_{j\in A}(H\cap\ell_j)$ and with winding number $0$ around each of the points
    in $\bigcup_{j\notin A}(H\cap\ell_j)$. We will show that each of these $2^k$ paths lifts to a path in $Y_k\setminus Y_k'$ that starts from $z_0$ and ends at a different point in $F$; this will prove the desired statement.
    
    To study lifts of such paths $\gamma$, we observe that the intersection $H\cap W$ is parametrized by
    $z\mapsto \left(z,\frac{z}{i\sqrt{m}}\right)$ where $z$ ranges over
    $\mathbf{C}\setminus\left\{\frac{a_j+i\sqrt{m}b_j}{2}:j=1,\dots,k\right\}$. For
    all $z\in\mathbf{C}$, any element of $\left(\pi|_{Y_k}\right)^{-1}\left(z,\frac{z}{i\sqrt{m}}\right)$ is of the form $\left(z, \frac{z}{i\sqrt{m}}, d_1,\ldots,d_k \right)$,
    with
    \begin{eqnarray*}
    \begin{aligned}
d_j^2=Q_{m,P_j}\left(z,\frac{z}{i\sqrt{m}},0\right)&=2(-a_j+i\sqrt{m}b_j)z+a_j^2+mb_j^2\\
    &=2(-a_j+i\sqrt{m}b_j)(z-z_j)
    \end{aligned}
    \end{eqnarray*}
    and $z_j=\frac{a_j+i\sqrt{m}b_j}{2}$, for all $j=1,\ldots,k$. Crucially,
    $Q_{m,P_j}\left(z,\frac{z}{i\sqrt{m}},0\right)$ is a linear function of $z$ that
    vanishes exactly at $z_j$, so that the order of vanishing at $z_j$ is exactly $1$. 
    
    Thus, let $\gamma:[0,1]\rightarrow H\cap W$ be a closed path, beginning and ending at
    $(0,0)$. The path $\gamma$ lifts to a path $\widetilde{\gamma}$ in $\left(\pi|_{Y_k}\right)^{-1}(H\cap W)\subseteq Y_k\setminus Y_k'$, that starts at $z_0$ and ends at some point $\widetilde{z}$ in $F$. By the above, the $(2+j)$-th coordinate $d_j(t)$ of $\widetilde{\gamma}(t)$ is a continuous square root of $Q_{m,P_j}\big(\gamma(t),0\big)$.
\begin{itemize}
    \item If $\gamma$ has winding number $1$ around the point $H\cap\ell_j=\left(z_j, \frac{z_j}{i\sqrt{m}}\right)$, then, viewing the argument of $Q_{m,P_j}\big(\gamma(t),0\big)$ as a continuous function of $t$, it follows (by the formula for $Q_{m,P_j}\left(z,\frac{z}{i\sqrt{m}},0\right)$ described above) that the arguments at $t=0$ and $t=1$ differ by $2\pi$. Therefore, any continuous square root of $Q_{m,P_j}\big(\gamma(t),0\big)$ will have arguments that differ by $\pi$ at $t=0$ and $t=1$. Since   $d_j(0)=\sqrt{a_j^2+mb_j^2}$ 
    (as this is the $(2+j)$-th coordinate of $z_0$), we deduce that $d_j(1)=-\sqrt{a_j^2+mb_j^2}$.
    \item If $\gamma$ has winding number $0$ around the point $H\cap\ell_j=\left(z_j, \frac{z_j}{i\sqrt{m}}\right)$, then, viewing the argument of $Q_{m,P_j}\big(\gamma(t),0\big)$ as a continuous function of $t$, we see that the arguments at $t=0$ and at $t=1$ are the same. Therefore, $d_j(1)=d_j(0)=\sqrt{a_j^2+mb_j^2}$.
\end{itemize}
In other words, if, for some $A\subset \{1,\ldots,k\}$, $\gamma$ has winding number 1 around each of the points in $\bigcup_{j\in A}(H\cap \ell_j)$, and winding number 0 around each of the points in $\bigcup_{j\notin A}(H\cap \ell_j)$, then $\gamma$ lifts to a path $\tilde\gamma$ in $Y_k\setminus Y_k'$ connecting $z_0$ to 
the point
\begin{equation*}
        \left(0,0,(-1)^{1_A(1)}\sqrt{a_1^2+mb_1^2},\dots,(-1)^{1_A(k)}\sqrt{a_k^2+mb_k^2}\right)\in F,
    \end{equation*} 
    where $1_A$ denotes the indicator function of $A$. Since such $\gamma$ exists for arbitrary $A\subset \{1,\ldots,k\}$, it follows that all elements of $F$ are in the same path-component, completing the proof of the lemma.
\end{proof}

This completes the proof of Lemma~\ref{lem:generalXirred}, and thus the proof of
Lemma~\ref{lem:generalX}.

\section{Degree and irreducibility of $\overline{\mathcal{C}_k}$}\label{sec:ag2}

In this section, we prove that $\overline{\mathcal{C}_k}$ is an irreducible curve of large
degree, treating the case that $C$ is not a line or circle and the case that $C$ is a line
or circle separately. Before stating our main results, we will require a preparatory
technical lemma.
\begin{lemma}\label{lem:Rpolys}
  Let $m\in\mathbf{N}$ and $Q\in\mathbf{C}[x,y]$ be a polynomial of degree $d\geq 3$, and write
  \begin{equation*}
    Q(x,y)=\sum_{j=0}^dQ_j(x,y),
  \end{equation*}
  where $Q_j$ is the degree $j$ homogeneous component of $Q$. Then,
  \begin{equation*}
    \overline{Q}(\mp i\sqrt{m}+wz,1,z)=:\sum_{j=0}^d R_{j,\pm}(w)z^j
  \end{equation*}
  where each $R_{j,\pm}$ is a polynomial of degree at most $j$ and, for some $0\leq j\leq d-2$, at
  least one of $R_{j,+}(w)$ or $R_{j,-}(w)$ is not identically zero.
\end{lemma}
\begin{proof}
  A short computation yields
  \begin{equation*}
    R_{j,\pm}(w)=\sum_{k=0}^j\left(\frac{d^k}{dx^k}Q_{d-j+k}(x,1)\Bigg|_{x=\mp i\sqrt{m}}\right)\frac{w^k}{k!},
  \end{equation*}
  for each $0\leq j\leq d$, which shows that $\deg{R_{j,\pm}}\leq j$. If
  $R_{j,+},R_{j,-}\equiv 0$ for all $0\leq j\leq d-2$, then  we must have
  \begin{equation*}
    \frac{d^k}{dx^k}Q_{\ell}(x,1)\Bigg|_{x=i\sqrt{m}}=\frac{d^k}{dx^k}Q_{\ell}(x,1)\Bigg|_{x=-i\sqrt{m}}=0
  \end{equation*}
  for all $\ell=2,\dots,d$ and $k=0,\dots,\ell-2$. When $\ell=d$, this implies that
  $i\sqrt{m}$ and $-i\sqrt{m}$ are both roots of $Q_d(x,1)$ of multiplicity at least
  $d-1$. Thus, $(x^2+m)^{d-1}\mid Q_d(x,1)$. But,
  $\deg{Q_d(x,1)}=d$ since $\deg{Q}=d$, so we must have $2(d-1)\leq d$, which is
  impossible when $d\geq 3$. We conclude that $R_{j,+}\not\equiv 0$ or
  $R_{j,-}\not\equiv 0$ for some $0\leq j\leq d-2$.
\end{proof}
\begin{lemma}\label{lem:generalC}
  Let $k\in\mathbf{N}$, $m\in\mathbf{N}$ be squarefree, and $C\subset\mathbf{C}^2$ be an
  irreducible curve that contains the origin and is not a line or circle such that the
  curve
    \begin{equation*}
    C':=\{(x,y):(x,y\sqrt{m})\in C\}  
    \end{equation*}
    is defined over $\mathbf{Q}$ and the origin is not a singular point of $C'$. Let
    $Q\in\mathbf{Q}[x,y]$ be such that $C'=V(Q)$ and $R_{j,\pm}$ for $0\leq j\leq \deg{Q}$
    be as in Lemma~\ref{lem:Rpolys}. Denote the set of singular points of $C'$ by
    $\mathcal{S}$ and define $\tilde{R}_{j,\pm}$ by
    $\tilde{R}_ {j,\pm}(x,y):=R_{j,\pm}(x\pm i y)$ for $0\leq j\leq\deg{Q}$. Let
    $P_1=(a_1,b_1),\dots,P_k=(a_k,b_k)\in C'\cap\mathbf{Q}^2$ be $k$ distinct points, none
    of which is the origin, a singular point of $C'$, or, when $\deg{Q}\geq 3$, a point in
    $\bigcap_{j=0}^{d-2}V(\tilde{R}_{j,+},\tilde{R}_{j,-})$, such that for every
    $1\leq j\neq j'\leq k$ the point
    \begin{equation*}
    \left(\frac{a_j+a_{j'}}{2}+i\sqrt{m}\frac{b_j-b_{j'}}{2},\frac{b_j+b_{j'}}{2}-i\frac{a_j-a_{j'}}{2\sqrt{m}}\right)  
    \end{equation*}
    is not on $C'$ and such that
    \begin{equation}
      \label{eq:notsingular}
      \left\{s\pm i\sqrt{m}t:(s,t)\in\mathcal{S}\right\}\cap\left\{a_j\pm i\sqrt{m} b_j:j=1,\dots,k\right\}=\emptyset.
    \end{equation}
    Define the affine variety $\mathcal{C}_k\subset\mathbf{C}^{k+2}$ by
    \begin{equation*}
        \mathcal{C}_k:=\left\{(x,y,d_1,\dots,d_k)\in\mathbf{C}^{k+2}:(x,y)\in C'\text{ and }Q_{m,P_j}(x,y,d_j)=0\text{ for all }i=1,\dots,k\right\}.
    \end{equation*}
    Then $\overline{\mathcal{C}_k}$ is an irreducible curve in $\mathbf{P}^{k+2}$ of
    degree at least $2^k$ and at most $2^k\deg{C}$, defined over $\mathbf{Q}$.
\end{lemma}

\begin{lemma}\label{lem:Clineorcirc}
  Let $k\in\mathbf{N}$, $m\in\mathbf{N}$ be squarefree, and $C\subset\mathbf{C}^2$ be a
  line or circle that contains the origin such that the curve
    \begin{equation*}
    C':=\left\{(x,y):(x,y\sqrt{m})\in C\right\}  
    \end{equation*}
    is defined over $\mathbf{Q}$. Let
    $P_1=(a_1,b_1),\dots,P_k=(a_k,b_k)\in \mathbf{Q}^2\setminus C'$ be $k$ distinct points
    such that, for every $1\leq j\neq j'\leq k$, the point
    \begin{equation*}
    \left(\frac{a_j+a_{j'}}{2}+i\sqrt{m}\frac{b_j-b_{j'}}{2},\frac{b_j+b_{j'}}{2}-i\frac{a_j-a_{j'}}{2\sqrt{m}}\right)  
    \end{equation*}
    is not on $C'$ and such that, if $C$ is a circle, none of
    $(a_1,\sqrt{m}b_1),\dots,(a_k,\sqrt{m}b_k)$ is the center of $C$. Define the affine
    variety $\mathcal{C}_k\subset\mathbf{C}^{k+2}$ by
    \begin{equation*}
        \mathcal{C}_k:=\left\{(x,y,d_1,\dots,d_k)\in\mathbf{C}^{k+2}:(x,y)\in C'\text{ and }Q_{m,P_j}(x,y,d_j)=0\text{ for all }i=1,\dots,k\right\}.
    \end{equation*}
    Then $\overline{\mathcal{C}_k}$ is an irreducible curve in $\mathbf{P}^{k+2}$ of
    degree at least $2^k$ and at most $2^k\deg{C}$, defined over $\mathbf{Q}$.
\end{lemma}

\subsection{Dimension and degree of $\overline{\mathcal{C}_k}$}

Given the results of the previous section, it is now easy to deduce the following.

\begin{lemma}\label{lem:Cdimdeg}
  Let $\mathcal{C}_k$ and $C$ be as in Lemma~\ref{lem:generalC} or
  Lemma~\ref{lem:Clineorcirc}. Then $\overline{\mathcal{C}_k}$ is a curve of degree at
  least $2^k$ and at most $2^k\deg{C}$.
\end{lemma}
\begin{proof} 
  Let $\pi:\mathbf{C}^{k+2}\rightarrow\mathbf{C}^2$ denote the standard projection map
  onto the first two coordinates. First of all, we certainly have that
  $\dim{\overline{\mathcal{C}_k}}\geq 1$, since
  $\dim{\overline{\mathcal{C}_k}}=\dim\mathcal{C}_k$ and $\pi(\mathcal{C}_k)=C'$. To see
  that $\dim{\overline{\mathcal{C}_k}}\leq 1$, note that $\mathcal{C}_k$ is a subvariety
  of the irreducible affine surface $X_k$ (where $X_k$ is defined as in
  Lemma~\ref{lem:generalX}), so that $\dim{\mathcal{C}_k}\leq 2$ and
  $\dim{\mathcal{C}_k}=2$ would imply that $\mathcal{C}_k=X_k$ (by the irreducibility of
  $X_k$), which is impossible since $\pi(\mathcal{C}_k)=C'\neq \mathbf{C}^2=\pi(X_k)$. We
  conclude that $\overline{\mathcal{C}_k}$ is a curve.

 The upper bound $$\deg{\overline{\mathcal{C}_k}}\leq \deg{C} \cdot \prod_{j=1}^k \deg{Q_{m,P_j}}\leq 2^k\deg{C}$$ is an immediate consequence of B\'ezout's inequality (Lemma~\ref{lem:Bezout}), since $\deg{\overline{Q}}=\deg{Q}=\deg{C}$. Showing that $\deg{\overline{\mathcal{C}_k}}\geq 2^k$ requires essentially the same
    argument as in the proof of Lemma~\ref{lem:generalXdeg}, since we assumed that $C$
    (and thus $C'$) contains the origin. Let $\ell=\{(x,y)\in\mathbf{C}^2:ax+by=0\}$ be a
    line that is not equal to $C'$.  Since the irreducible curve $C'$ is not $\ell$, $\ell\cap C'$ is finite. Consider the subspace $H\subset\mathbf{P}^{k+2}$ defined by $H=V(ax+by)$. Then
    \begin{equation*}
        \overline{\mathcal{C}_k}\cap H\supset \left\{\left[0:0:1:\pm\sqrt{a_1^2+mb_1^2}:\dots:\pm\sqrt{a_k^2+mb_k^2}\right]\right\},
    \end{equation*}
    where the set on the right-hand side has $2^k$ elements, since none of $P_1,\dots,P_k$
    is the origin and thus $$\sqrt{a_j^2+mb_j^2}\neq 0,\quad j=1,\dots,k.$$ 
    Moreover, $\overline{\mathcal{C}_k}\cap H$ is finite. Indeed, since $\ell\cap C'$ is finite, the definition of $H$ gives
    \begin{equation*}
        \overline{\mathcal{C}_k}\cap H\subset \overline{X_k}\cap \overline{C'\times \mathbf{C}^k}\cap H=\overline{X_k}\cap\bigcup_{j=1}^nH_j,
    \end{equation*}
    for some finite number of affine hyperplanes $H_1,\dots,H_n$, each of the form
    \begin{equation*}
    H_j=\left\{[x:y:z:d_1:\dots:d_k]: [x:y:z]=w_j\right\}  
    \end{equation*}
    for some fixed $w_j\in\mathbf{P}^2$ (a point in the finite set $\overline{\ell}\cap \overline{C'}$). Since, for each $w \in\mathbf{P}^2$, the fiber
    \begin{equation*}
    \left(\pi|_{\overline{X_k}}\right)^{-1}(w)=\left\{ [x:y:z:d_1:\dots:d_k]\in\overline{X_k}: [x:y:z]=w\right\}  =\overline{X_k}\cap H_j
    \end{equation*}
in $\overline{X_k}$ above $w$  is finite, it follows that each $\overline{X_k}\cap H_j$ is finite, hence $\overline{\mathcal{C}_k}\cap H$ is finite. Thus, by Lemma~\ref{lem:deglow}, $\deg{\overline{\mathcal{C}_k}}\geq 2^k$.
\end{proof}

\subsection{Irreducibility of $\overline{\mathcal{C}_k}$}

Now, we will show the following.
\begin{lemma}\label{lem:generalCirred}
  Let $\overline{\mathcal{C}_k}$ be as in Lemma~\ref{lem:generalC} or
  Lemma~\ref{lem:Clineorcirc}. Then $\overline{\mathcal{C}_k}$ is irreducible.
\end{lemma}
As in the proof of Lemma~\ref{lem:generalXirred}, it suffices to check that
$\mathcal{C}_k$ is an irreducible affine variety. We will, similarly, begin this section
by roughly describing the set of singular points of $\mathcal{C}_k$.

\begin{lemma}\label{lem:singC}
    Let $\mathcal{C}_k$ and $\mathcal{S}$ be as in Lemma~\ref{lem:generalC}. Denote the set of singular
    points of $\mathcal{C}_k$ by $\mathcal{C}'_k$. Then $\mathcal{C}'_k$ is finite,
    \begin{equation}\label{eq:singC2}
    \mathcal{C}'_k\supset\left\{(x,y,d_1,\dots,d_k)\in\mathcal{C}_k:(x,y)\in\{P_1,\dots,P_k\}\cup \mathcal{S}\right\},
  \end{equation}
  and, if $(x,y)\in \pi (\mathcal{C}'_k)$, then
  \begin{equation}\label{eq:singC3}
    \mathcal{C}_k'\supset \mathcal{C}_k\cap\left\{(x,y,d_1,\dots,d_k):d_1,\dots,d_k\in\mathbf{C}\right\}.
  \end{equation}
\end{lemma}
\begin{proof}
  First of all, it is a standard fact that the set of singular points of an algebraic
  curve is finite (see, for example, Theorem~5.3 of~\cite[Chapter I]{Hartshorne77}).

  Writing $C'=V(Q(x,y))\subset\mathbf{C}^2$ for an appropriate polynomial
  $Q\in\mathbf{Q}[x,y]$, we have that the Jacobian matrix $J(P)$ of $\mathcal{C}_k$ at a point
  $P=(x,y,d_1,\dots,d_k)$ is
    \begin{equation}\label{eq:jacobian}
      J(P)=\begin{pmatrix}
        \frac{\partial}{\partial x}Q(x,y) & \frac{\partial}{\partial y}Q(x,y) & 0 & \dots & \dots & 0 \\
        2(x-a_1) & 2m(y-b_1) & -2d_1 & 0 & \dots & 0 \\
        2(x-a_2) & 2m(y-b_2) & 0 & -2d_2 & \dots & 0 \\
        \vdots & \vdots & \vdots & \vdots & \ddots & \vdots \\
        2(x-a_k) & 2m(y-b_k) & 0 & \dots & 0 & -2d_k
        \end{pmatrix}.
    \end{equation}
    If $P=(x,y,d_1,\dots,d_k)\in\mathcal{C}_k$ with
    $(x,y)\in\{P_1,\dots,P_k\}\cup \mathcal{S}$, then some row of $J(P)$ will
    be the zero vector. This means that $J(P)$ can have rank at most $k$, so
    that $P$ must be a singular point of $\mathcal{C}_k$. This proves \eqref{eq:singC2}.

    Finally, if $P=(x,y,d_1,\dots,d_k)\in\mathcal{C}_k'$, then, by definition, $J(P)$ has
    rank at most $k$. Since $\mathcal{C}_k\subset X_k$, by the definition \eqref{eq:Xk} of $X_k$, any element
    $P'$ of
    \begin{equation*}
    \mathcal{C}_k\cap\left\{(x,y,\tilde{d}_1,\dots,\tilde{d}_k):\tilde{d}_1,\dots,\tilde{d}_k\in\mathbf{C}\right\}    
    \end{equation*}
    has the form $P'=(x,y,(-1)^{\epsilon_1}d_1,\dots,(-1)^{\epsilon_k}d_k)$ for some
    $(\epsilon_1,\dots,\epsilon_k)\in\{0,1\}^k$. Then $J(P')=J(P)A$ for the
    $(k+2)\times(k+2)$ diagonal matrix $A$ that has entries
    $1,1,(-1)^{\epsilon_1},\dots,(-1)^{\epsilon_k}$ along the diagonal. Since $A$ is
    nonsingular, $J(P')$ must have the same rank as $J(P)$. It follows that
    $P'\in\mathcal{C}_k'$ as well, proving \eqref{eq:singC3}.
\end{proof}
Arguing in almost the exact same manner (the only differences being that $C'$ is smooth
and $P_1,\dots,P_k\notin C'$) yields the analogous lemma for when $C$ is a line or circle.
\begin{lemma}\label{lem:singClineorcirc}
    Let $\mathcal{C}_k$ and $C'$ be as in Lemma~\ref{lem:Clineorcirc}. Denote the set of singular
    points of $\mathcal{C}_k$ by $\mathcal{C}'_k$. Then $\mathcal{C}'_k$ is finite and, if $(x,y)\in \pi (\mathcal{C}'_k)$, then
  \begin{equation}\label{eq:singC3lineorcirc}
    \mathcal{C}_k'\supset \mathcal{C}_k\cap\left\{(x,y,d_1,\dots,d_k):d_1,\dots,d_k\in\mathbf{C}\right\}.
  \end{equation}
\end{lemma}

Next, we will analyze the intersection of $C'$ with the lines
$\ell_1,\ell_1',\dots,\ell_k,\ell_k'$ defined in the proof of Lemma~\ref{lem:singularpts}.
\begin{lemma}\label{lem:zj}
  Let $C'$ and $P_1,\dots,P_k$ be as in Lemma~\ref{lem:generalC} and, for each
  $j=1,\dots,k$, define
  \begin{equation*}
    \ell_j : x+i\sqrt{m}y=a_j+i\sqrt{m}b_j
  \end{equation*}
  and
  \begin{equation*}
   \ell_j' : x-i\sqrt{m}y=a_j-i\sqrt{m}b_j.
  \end{equation*}
  Then, for each $j=1,\dots,k$ we have
  \begin{enumerate}
      \item $|C'\cap\ell_j|,|C'\cap\ell_j'|<\infty$,
      \item $C'\cap\ell_j\cap\ell_j'=\{(a_j,b_j)\}$, and
      \item there exists a point
        \begin{equation*}
        z_{j}\in C'\cap\left(\ell_j\cup\ell_j'\right)  
        \end{equation*}
        not equal to $(a_j,b_j)$ or any of the singular points of $C'$.
  \end{enumerate}
\end{lemma}
\begin{proof}
  We certainly have $|C'\cap\ell_j|,|C'\cap\ell_j'|<\infty$ for each $j=1,\dots,k$, since
  $C'$ is an irreducible curve that is not a line. In the proof of
  Lemma~\ref{lem:singularpts}, we saw that $\ell_j\cap\ell_j'=\{(a_j,b_j)\}$ for each
  $j=1,\dots,k$, which gives~(2) since $(a_j,b_j)=P_j\in C'$.

  To prove~(3), first note that $\overline{C'}=V(\overline{Q})$. By B\'ezout's theorem,
  the number of intersections, counted with multiplicity, of $\overline{C'}$ with
  $\overline{\ell_j}$ and of $\overline{C'}$ with $\overline{\ell_j'}$ are both
  $d:=\deg{Q}$. Letting $n_1,n_2,n_1',$ and $n_2'$ denote the intersection multiplicity of
  $\overline{C'}$ with $\overline{\ell_j}$ at $[a_j:b_j:1]$ and at a point at infinity and
  with $\overline{\ell_j'}$ at $[a_j:b_j:1]$ and at a point at infinity, respectively, it
  thus suffices to show that $n_1+n_2<d$ or $n_1'+n_2'<d$.

  Let us first consider the intersection multiplicities $n_2$ and $n_2'$ of
  $\overline{C'}$ with $\overline{\ell_j}$ and of $\overline{C'}$ with
  $\overline{\ell_j'}$, respectively, at infinity when $d\geq 3$. The point at infinity
  lying on $\overline{\ell_j}$ is $[-i\sqrt{m}:1:0]$ and the point at infinity lying on
  $\overline{\ell_j'}$ is $[i\sqrt{m}:1:0]$. On $\{[x:1:z]: x,z \in \mathbf{C}\}$,
  $\overline{\ell_j}$ and $\overline{\ell_j'}$ can be parametrized as
  \[\left\{[-i\sqrt{m}+(a_j+i\sqrt{m}b_j)z:1:z]: z \in \mathbf{C}\right\}\] and
  \[\left\{[i\sqrt{m}+(a_j-i\sqrt{m}b_j)z:1:z]: z \in \mathbf{C}\right\},\] respectively. Thus, $n_2$
  is the multiplicity of $z=0$ as a root of
  $Q_{+}(z)=\overline{Q}(-i\sqrt{m}+(a_j+i\sqrt{m}b_j)z, 1, z)$ and $n_2'$ is the
  multiplicity of $z=0$ as a root of
  $Q_{-}(z)=\overline{Q}(i\sqrt{m}+(a_j-i\sqrt{m}b_j)z, 1, z)$. It then follows from
  Lemma~\ref{lem:Rpolys} and the assumption that
  $(a_j,b_j)\notin \bigcap_{l=0}^{d-2}V(\tilde{R}_{l,+},\tilde{R}_{l,-})$ that there
  exists a $0\leq l\leq d-2$ such that the coefficient of $z^{l}$ in $Q_{+}(z)$ or
  $Q_{-}(z)$ is nonzero. Thus, $n_2$ or $n_2'$ is at most $d-2$ when $d\geq 3$.
  
  Now, note that $n_1$ (respectively, $n_1'$) is the intersection multiplicity of the line
  $x+i\sqrt{m}y=0$ (respectively, $x-i\sqrt{m}y=0$) with the curve defined by the
  vanishing of the polynomial $Q_{a_j,b_j}(x,y)=Q(x+a_j,y+b_j)$ at the point
  $(0,0)$. Thus, $n_1$ (respectively, $n_1'$) is the multiplicity of the root $y=0$ of the
  polynomial $\tilde Q_{+}(y)=Q_{a_j,b_j}(-i\sqrt{m}y,y)$ (respectively,
  $\tilde Q_{-}(y)=Q_{a_j,b_j}(i\sqrt{m}y,y)$). If $n_1$ (respectively, $n_1'$) were greater than
  $1$, we would have that if $\tilde Q_{+}(y)=\sum_{t=0}^{d}c_ty^t$ (respectively,
  $\tilde Q_{-}(y)=\sum_{t=0}^dc_t'y^t$), then $c_0=c_1=0$ (respectively, $c_0'=c_1'=0$). But,
  observe that, as a function of $(a_j,b_j)$, the linear term of $Q_{a_j,b_j}(x,y)$ equals
    \begin{equation*}
      \frac{\partial Q}{\partial x}(a_j,b_j)\cdot x+\frac{\partial Q}{\partial y}(a_j,b_j)\cdot y,
    \end{equation*}
    so that $c_1=c_1'=0$ if and only if both
    \begin{equation}\label{eq:zero1}
      -i\sqrt{m}\frac{\partial Q}{\partial x}(a_j,b_j)+\frac{\partial Q}{\partial y}(a_j,b_j)=0,
    \end{equation}
    and
    \begin{equation}\label{eq:zero2}
      i\sqrt{m}\frac{\partial Q}{\partial x}(a_j,b_j)+\frac{\partial Q}{\partial y}(a_j,b_j)=0.
    \end{equation}
    But, since $Q$ has rational coefficients, neither~\eqref{eq:zero1}
    nor~\eqref{eq:zero2} can hold since, by hypothesis, $(a_j,b_j)$ is not a singular point of
    $C'$. We therefore conclude that $n_1=n_1'=1$.

    Thus, if $d\geq 3$, then $n_1+n_2$ or $n_1'+n_2'$ is less than $d$, from which (3)
    follows. Since $d\neq 1$ by the assumption that $C$ is not a line, it only remains to
    rule out that $n_1+n_2,n_1'+n_2'=d$ when $d=2$.
    
    In the case $d=2$, $n_1+n_2=n_1'+n_2'=d$ would imply that $n_1=n_2=n_1'=n_2'=1$.  Write
    $Q(x,y)=Q_2(x,y)+Q_1(x,y)+Q_0$ where $Q_j$ is the degree $j$ homogeneous component of
    $Q$ for $j=0,1,2$. Then $n_2=n_2'=1$ implies that $Q_2(\mp i\sqrt{m},1)=0$, which
    means that $Q_2(x,1)=x^2+m$ and, thus, that the degree $2$ homogeneous component of
    $Q(x,y/\sqrt{m})$ must equal $x^2+y^2$. This forces $C$ to either be a circle or the
    union of two lines, both of which contradict our assumptions on $C$ in
    Lemma~\ref{lem:generalC}.

    We have thus shown that there exists a point $z_j\in C'\cap(\ell_j\cup\ell_j')$ not
    equal to $(a_j,b_j)$ for each $j=1,\dots,k$. That $z_j$ is, additionally, not a
    singular point of $C'$ follows immediately from the assumption~\eqref{eq:notsingular}
    in Lemma~\ref{lem:generalC}, which says that none of the singular points of $C'$ lie
    in $\bigcup_{j=1}^k(\ell_j\cup\ell_j')$. This completes the proof of the lemma.
\end{proof}

\begin{lemma}\label{lem:zjlineorcirc}
  Let $C'$ and $P_1,\dots,P_k$ be as in Lemma~\ref{lem:Clineorcirc} and, for each
  $j=1,\dots,k$, define
  \begin{equation*}
    \ell_j : x+i\sqrt{m}y=a_j+i\sqrt{m}b_j
  \end{equation*}
  and
  \begin{equation*}
   \ell_j' : x-i\sqrt{m}y=a_j-i\sqrt{m}b_j.
  \end{equation*}
  Then, for each $j=1,\dots,k$ we have
  \begin{enumerate}
      \item $|C'\cap\ell_j|,|C'\cap\ell_j'|<\infty$,
      \item $C'\cap\ell_j\cap\ell_j'=\emptyset$, and
      \item there exists a point
        \begin{equation*}
        z_{j}\in C'\cap\left(\ell_j\cup\ell_j'\right).  
        \end{equation*}
  \end{enumerate}
\end{lemma}
\begin{proof}
  Since $\ell_j\cap\ell_j'=\{(a_j,b_j)\}$ for all $j=1,\dots,k$, we certainly always have
  (2), i.e., $C'\cap\ell_j\cap\ell_j'=\emptyset$.

  For the other two claims, we begin with the case that $C$, and thus $C'$, is a line, so
  that $C'=V(Q(x,y))$ where $Q(x,y)=ax+by-c$ for $a,b,c\in\mathbf{Q}$. Then no two of
  $\{C',\ell_j,\ell_j'\}$ are parallel, so that $|C'\cap\ell_j|=|C'\cap\ell_j'|=1$,
  yielding~(1) and~(3).
    
  Now we deal with the case that $C$ is a circle, so that $C'=V(Q(x,y))$ where
  $Q(x,y)=(x-a)^2+m(y-b)^2-r$ with $a,b,r\in\mathbf{Q}$. As in the proof of the previous
  lemma, we certainly have $|C'\cap\ell_j|,|C'\cap\ell_j'|<\infty$ since $C'$ is
  irreducible and not a line. Also as in the proof of the previous lemma, we have that the
  number of intersections (counted with multiplicity) of $\overline{C'}$ with
  $\overline{\ell_j}$ and of $\overline{C'}$ with $\overline{\ell_j'}$ are both equal to
  $\deg{Q}=2$. So, to conclude~(3) in this case, we just need to show that it is
  impossible for the intersection multiplicities $n$ of $\overline{C'}$ with
  $\overline{\ell_j}$ at $[-i\sqrt{m}:1:0]$ (the point at infinity lying on
  $\overline{\ell_j}$) and $n'$ of $\overline{C'}$ with $\overline{\ell'_j}$ at
  $[i\sqrt{m}:1:0]$ (the point at infinity lying on $\overline{\ell_j'}$) to both equal
  $2$. On $\{[x:1:z]: x,z \in \mathbf{C}\}$, $\overline{\ell_j}$ and $\overline{\ell_j'}$
  can be parametrized as
  \[\left\{[-i\sqrt{m}+(a_j+i\sqrt{m}b_j)z:1:z]: z \in \mathbf{C}\right\}\] and
  \[\left\{[i\sqrt{m}+(a_j-i\sqrt{m}b_j)z:1:z]: z \in \mathbf{C}\right\},\] respectively. Thus, $n$
  is the multiplicity of $z=0$ as a root of
  $Q_{+}(z):=\overline{Q}(-i\sqrt{m}+(a_j+i\sqrt{m}b_j)z, 1, z)$ and $n'$ is the
  multiplicity of $z=0$ as a root of
  $Q_{-}(z):=\overline{Q}(i\sqrt{m}+(a_j-i\sqrt{m}b_j)z, 1, z)$. If $n$ (respectively,
  $n'$) were equal to $2$, then $Q_+(z)$ (respectively, $Q_-(z)$) would be equal to a
  multiple of $z^2$. But, expanding out the definition of $Q_+(z)$ (respectively,
  $Q_-(z)$) yields
  \begin{equation*}
      Q_+(z)=\left((a_j+i\sqrt{m}b_j-a)^2+mb^2-r\right)z^2+\left(-2 i\sqrt{m}(a_j-a)+2m(b_j-b)\right)z
    \end{equation*}
    (respectively,
    \begin{equation*}
      Q_-(z)=\left((a_j-i\sqrt{m}b_j-a)^2+mb^2-r\right)z^2+\left(2 i\sqrt{m}(a_j-a)+2m(b_j-b)\right)z\text{)},
    \end{equation*}
    which is of the form $uz^2+vz$ for some nonzero $u,v\in\mathbf{C}$ by the assumption
    that $(a_j,\sqrt{m}b_j)$ is not the center of $C$ (which implies that
    $(a_j,b_j)\neq (a,b)$). Thus, we must have $n=n'=1$. So, in fact,
    $|C'\cap\ell_j|=|C'\cap\ell_j'|=1$ for all $j=1,\dots,k$, which implies~(3)
    in this case as well.
\end{proof}

Let $\mathcal{C}_k$ be as in Lemma~\ref{lem:generalC} or Lemma~\ref{lem:Clineorcirc}, and
set $\mathcal{D}_k:=\mathcal{C}_k\setminus\mathcal{C}'_k$ to be the set of nonsingular
points of $\mathcal{C}_k$. Analogously to our argument\footnote{See the discussion between Lemma \ref{lem:singularpts} and Lemma \ref{lem:YnoY'}.} in Section~\ref{sec:ag}, to show
that $\mathcal{C}_k$ is irreducible, it suffices to show that $\mathcal{D}_k$ is
irreducible, and since $\mathcal{D}_k$ is a smooth quasiaffine variety, it further
suffices to show that $\mathcal{D}_k$ is connected. As in Section~\ref{sec:ag.irr}, we
define
\begin{equation*}
\mathcal{D}'_k:=\left\{(x,y,d_1,\dots,d_k)\in \mathcal{D}_k:Q_{m,P_j}(x,y,0)=0\text{ for some }j=1,\dots,k\right\},  
\end{equation*}
so that
\begin{equation*}
 \mathcal{D}_k\setminus \mathcal{D}'_k=\mathcal{D}_k\setminus\bigcup_{j=1}^k\pi^{-1}(\ell_j\cup\ell'_j)=\mathcal{C}_k\setminus \left(\mathcal{C}'_k\cup \bigcup_{j=1}^k\pi^{-1}(\ell_j\cup\ell'_j)\right).
\end{equation*}
Since $\mathcal{D}_k\setminus\mathcal{D}_k'$ is dense in $\mathcal{D}_k$ (as all irreducible components of $\mathcal{C}_k$ are one-dimensional), it suffices
to show that $\mathcal{D}_k\setminus\mathcal{D}_k'$ is connected, and, analogously to the
argument in Section~\ref{sec:ag}, it further suffices to show that
$\mathcal{D}_k\setminus\mathcal{D}_k'$ is connected in the Euclidean topology on $\mathbf{C}^{k+2}$, which we will work in for the remainder of this section.
\begin{lemma}\label{lem:Dconnect}
  The set $\mathcal{D}_k\setminus \mathcal{D}'_k$ is path connected.
\end{lemma}

\begin{proof}
  We argue along the lines of the proof of Lemma~\ref{lem:YnoY'}. Observe that
  \begin{equation*}
\pi :\mathcal{D}_k\setminus\mathcal{D}_k'\to C'\setminus \left(\pi(\mathcal{C}_k')\cup\bigcup_{j=1}^k(\ell_j\cup\ell_j')\right)=:C''
  \end{equation*}
  is a $2^k$-fold covering map by claim \eqref{eq:singC3} of Lemma~\ref{lem:singC} or claim~\eqref{eq:singC3lineorcirc} of Lemma~\ref{lem:singClineorcirc}. Since, by
  Lemmas~\ref{lem:singC} and~\ref{lem:zj} or by Lemmas~\ref{lem:singClineorcirc} and~\ref{lem:zjlineorcirc}, the set $$C'\setminus C''=C'\cap
  \left(\pi(\mathcal{C}_k')\cup\bigcup_{j=1}^k(\ell_j\cup\ell_j')\right)$$ is finite, by
  applying the normalization theorem for irreducible algebraic curves to
  $\overline{C'}\subset\mathbf{P}^2$ and using that any compact Riemann surface is
  homeomorphic to either the $2$-sphere or a finite connected sum of $2$-tori (see,
  respectively,~\cite{Griffiths1989} and~\cite{Donaldson2011}, for example), it follows
  that $C''$ is homeomorphic to either a $2$-sphere minus a finite number of points or a
  finite connected sum of $2$-tori minus a finite number of points, where there is a
  bijective correspondence between the finitely many points removed from the Riemann
  surface and the finitely many points in $\overline{C'}\setminus C''$. One consequence of this observation is that $C''$ is path-connected. So, analogously to the
  proof of Lemma~\ref{lem:YnoY'}, it thus suffices to
  show that the fiber
  $F:=\left(\pi|_{\mathcal{D}_k\setminus\mathcal{D}'_k}\right)^{-1}(\pi(z_0))$ is
  path-connected in $\mathcal{D}_k\setminus\mathcal{D}_k'$, where
  \begin{equation*}
    z_0:=\left(0,0,\sqrt{a_1^2+mb_1^2},\dots,\sqrt{a_k^2+mb_k^2}\right)\in \mathcal{D}_k\setminus \mathcal{D}'_k.
  \end{equation*}
 Note that 
  \begin{equation*}
   F=\left\{\left(0,0,(-1)^{\epsilon_1}\sqrt{a_1^2+mb_1^2},\dots,(-1)^{\epsilon_k}\sqrt{a_k^2+mb_k^2}\right):(\epsilon_1,\dots,\epsilon_k)\in\{0,1\}^k\right\}\subset\mathcal{D}_k\setminus \mathcal{D}'_k
  \end{equation*}
  since none of $P_1,\dots,P_k$ is the origin, $C'$ is nonsingular at the origin, and none
  of the lines $\ell_1,\ell_1',\dots,\ell_k,\ell_k'$ pass through the origin. By one of the
  assumptions of Lemma~\ref{lem:generalC} and Lemma~\ref{lem:Clineorcirc}, the set
\begin{equation*}
  (\ell_j\cup \ell_j')\cap (\ell_t\cup
    \ell_t')=\left\{\left(\tfrac{a_j+a_t}{2}+i\sqrt{m}\tfrac{b_j-b_t}{2},\tfrac{b_j+b_t}{2}-i\tfrac{a_j-a_t}{2\sqrt{m}}\right),\left(\tfrac{a_j+a_t}{2}-i\sqrt{m}\tfrac{b_j-b_t}{2},\tfrac{b_j+b_t}{2}+i\tfrac{a_j-a_t}{2\sqrt{m}}\right)\right\}
  \end{equation*}
  is disjoint from $C'$ for every  $1\leq j< t\leq k$, so the points
  \begin{equation*}
    z_j\in C'\cap\left(\ell_j\cup \ell_j'\right)\setminus \left(\mathcal{S}\cup \{(a_j,b_j)\}\right),\quad j=1,\dots,k,
  \end{equation*}
  or
  \begin{equation*}
    z_j\in C'\cap\left(\ell_j\cup \ell_j'\right),\quad j=1,\dots,k,
  \end{equation*}
  defined as in Lemma~\ref{lem:zj}~(3) or Lemma~\ref{lem:zjlineorcirc}~(3), respectively, are all distinct. A second consequence of our observation about the topology of $C''$ is that, for any
  choice of $J\subset\{1,\dots,k\}$, there exists a path $\gamma_J$ contained in $C''$
  starting and ending at $(0,0)$ with winding number $1$ around each of the points in
  $\{z_j:j\in J\}$ and with winding number $0$ around each of the points in $(C'\setminus C'')\setminus\{z_j:j\in J\}$. 
  
  Lifting the closed path $\gamma_J$ then  gives rise to a
  path in $\mathcal{D}_k\setminus \mathcal{D}'_k$ from $z_0$ to the point in
  $\left(\pi|_{\mathcal{D}_k}\right)^{-1}(\pi(z_0))$ with the property that, for $j\in J$, its $(2+j)$-th coordinate is the opposite of the $(2+j)$-th coordinate of $z_0$, i.e., equals
  \begin{equation*}
    -\sqrt{Q_{m,P_j}\left(0,0,0\right)}=-\sqrt{a_j^2+mb_j^2},
  \end{equation*}
while, for $j\in \{1,\dots,k\}\setminus J$, its $(2+j)$-th coordinate agrees with the $(2+j)$-th coordinate of $z_0$, i.e., equals
  \begin{equation*}  \sqrt{Q_{m,P_j}\left(0,0,0\right)}=\sqrt{a_j^2+mb_j^2}.
\end{equation*}
 In other words, $\gamma_J$ lifts from $C''$ to a path in $\mathcal{D}_k\setminus \mathcal{D}'_k$ connecting $z_0$ to the point
    \begin{equation*}
        \left(0,0,(-1)^{1_J(1)}\sqrt{a_1^2+mb_1^2},\dots,(-1)^{1_J(k)}\sqrt{a_k^2+mb_k^2}\right)\in \left(\pi|_{\mathcal{D}_k\setminus\mathcal{D}'_k}\right)^{-1}(\pi(z_0)).
    \end{equation*} 
     Since $J\subset\{1,\dots,k\}$ was arbitrary, it follows that all elements of $\left(\pi|_{\mathcal{D}_k\setminus\mathcal{D}'_k}\right)^{-1}(\pi(z_0))=\left(\pi|_{\mathcal{D}_k\setminus\mathcal{D}'_k}\right)^{-1}(0,0)$ are in the same path-component, completing the proof of the lemma. 
\end{proof}

This completes the proof of Lemma~\ref{lem:generalCirred}.  Combining
Lemmas~\ref{lem:Cdimdeg} and~\ref{lem:generalCirred} now yields Lemmas~\ref{lem:generalC} and~\ref{lem:Clineorcirc}.

\section{The intersections of integer distance sets with lines and circles}\label{sec:circles}

\subsection{Lines} Proposition~\ref{prop:lines} was essentially proved in
\cite[Section~5]{KurzWassermann2011}. However, since some of the terms used in \cite{KurzWassermann2011} are
nonstandard and \cite[Theorem~4]{KurzWassermann2011} is phrased in a manner that is not equivalent
to Proposition~\ref{prop:lines},  we will present a detailed proof of Proposition~\ref{prop:lines} here.

\begin{proof}[Proof of Proposition~\ref{prop:lines}] Let $K$ be the maximal number of points $P_1,\dots,P_K$ of $S$ that are all contained in a line $\ell$. Without loss of generality, we may assume that $\ell=\mathbf{R}\times \{0\}$ and  $P_i=(m_i,0)$ for $i=1,\dots,K$, where $m_1<m_2<\dots<m_K$. Let $Q=(q_1,q_2) \in S$ be such that the triangle with vertices $P_1$, $Q$, and $P_K$ has maximal area (note that by our assumption that not all the points of $S$ are collinear, this maximal area must be positive). Then, $|q_2|>0$ is the height of the triangle $\triangle=P_1QP_K$ from $Q=(q_1,q_2)$ to $(q_1,0)$ on the base $P_1P_K$ of $\triangle$. For $i=1,\dots,K-1$, set $a_i:=m_{i+1}-m_i$ and $a:=m_K-m_1$, so that $a_i,a\in\mathbf{N}$.  For $i=1,\dots,K$, let $b_i$ be the distance between $Q$ and $P_i$, so that $b_i\in\mathbf{N}$. See Figure \ref{fig:prop1.3}.
   \begin{figure}
    \centering
\centerline{\includegraphics[height=2.7in]{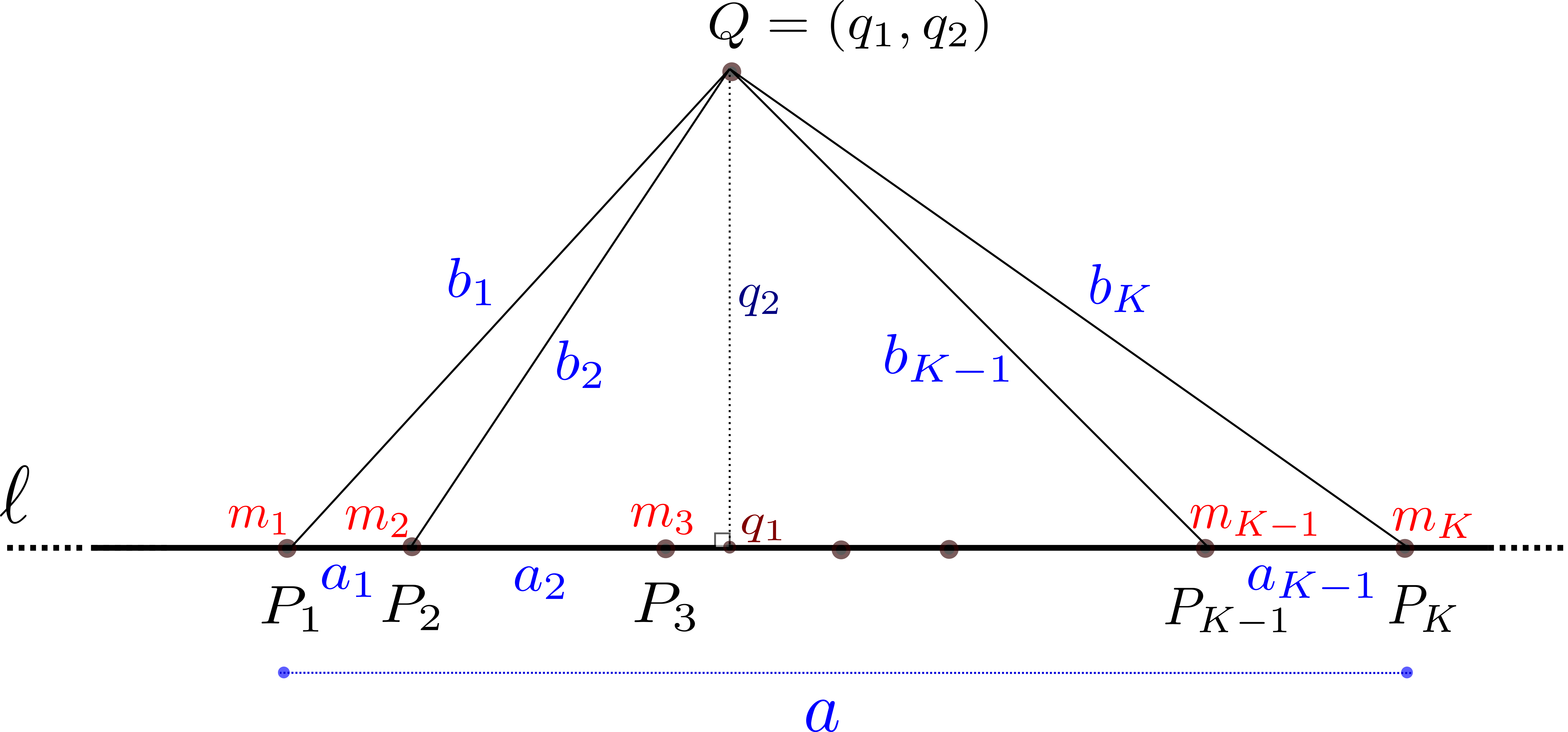}}
    \caption{The points $P_1=(m_1,0),\dots,P_K=(m_K,0)$ lie on the line $\ell=\mathbf{R}\times\{0\}$, and all the distances $a_1,\dots,a_{K-1},a,b_1,\dots,b_{K}$ are integers.}
    \label{fig:prop1.3}
\end{figure}
Then, by 
Pythagoras' theorem, we have $$q_2^2+(q_1-m_1)^2=b_1^2 \quad \text{ and } \quad q_2^2+\left(a-(q_1-m_1)\right)^2=b_K^2,$$
which, on subtracting one equation from the other, gives \begin{equation}\label{eq:cs}
    q_1-m_1=\frac{b_1^2-b_K^2+a^2}{2a}.
\end{equation}
Letting $D:=\frac{2a}{\left(b_1^2-b_K^2+a^2,2a\right)}\in\mathbf{Z}$, then, by \eqref{eq:cs}, $D(q_1-m_1)$ is also an integer; hence, for  $i=1,\dots,K$, since $m_i-m_1=\sum_{j=1}^{i-1}a_j$ is in $\mathbf{N}$, we have that $$D(q_1-m_i)=D(q_1-m_1)-D(m_i-m_1)$$ is an integer too.  A further application of Pythagoras' theorem gives 
$q_2^2=b_i^2-(q_1-m_i)^2$ for $i=1,\dots,K$. Thus, 
\begin{equation}\label{eq:decomposition}
    D^2q_2^2=\left[Db_i+D(q_1-m_i)\right]\left[Db_i-D(q_1-m_i)\right]
\end{equation}
for $i=1,\dots,K$. Note that we either have $m_{\lfloor K/2\rfloor}\leq q_1$ or $m_{\lfloor K/2\rfloor}> q_1$. If  the former holds, then for every $2\leq i\leq \lfloor K/2\rfloor$ we have  $b_{i-1}>b_i$ and $(q_1-m_{i-1})>(q_1-m_i)>0$; if the latter holds, then  for every $\lceil K/2\rceil \leq i\leq K$ we have  $b_{i-1}<b_i$ and $0<(m_{i-1}-q_1)<(m_i-q_1)$. We conclude that, in either case, 
$$\left|\left\{Db_i+D(q_1-m_i)\colon i=1,\dots,K\right\}\right|\asymp K.$$
So, since $Db_i$ and $D(q_1-m_i)$, $i=1,\dots,K$ are all integers, \eqref{eq:decomposition} implies that the natural number $D^2q_2^2$ has $\gg K$ different divisors, hence  $K \ll\tau(D^2q_2^2)$. Since $D^2q_2^2\ll N^4$, we thus have 
\begin{equation*}
    K\ll \tau (D^2 q_2^2)\ll (D^2 q_2^2)^{O\left(1/\log\log (D_2 q_2^2)\right)}\ll N^{O(1/\log\log N)},
\end{equation*}
using the standard estimate $\tau(n)\ll n^{O(1/\log\log n)}$ (see, for example,~\cite[Theorem 315]{HW79}), which gives the claim.
\end{proof}

\subsection{Circles} 
Recall that Proposition~\ref{prop:main} controls the size of an integer distance set in $[-N,N]^2$ that lives on an arbitrary circle. It turns out that the bulk of the argument lies in tackling the case where the circle is fully contained in $[-N,N]^2$; the general case easily follows. In particular, for any positive real number $r$, define $c(r)$ to be the maximum size of an integer distance set contained in the circle of radius $r$ centered at the origin. 
Then, this special case of Proposition~\ref{prop:main} is equivalent to the following.
\begin{proposition}\label{prop:mainalt}
    We have
    \begin{equation*}
        \max_{r\in[1,N]}c(r)\ll N^{O(1/\log\log{N})}.
    \end{equation*}
\end{proposition}
We can deduce Proposition \ref{prop:main} from Proposition \ref{prop:mainalt} as follows.

\begin{proof}[Proof of Proposition~\ref{prop:main}] Let $N\in\mathbf{N}$, $C$ be a circle, and $S\subset C$ be an integer distance set. Our goal is to show that
\begin{equation*}
    |S\cap [-N,N]^2|\ll N^{O(1/\log\log N)}.
\end{equation*}

 Let $c>0$ be an absolute constant to be fixed shortly. Since $N$ is assumed to be large, if $C$ has radius  at most $cN^2$, then $|S|\ll N^{O(1/\log\log N)}$ by Proposition~\ref{prop:mainalt}. Thus, the proof is complete in this case.

If $C$ has radius  $r\geq cN^2$, then  we claim that $|S\cap [-N,N]^2|\leq 2$, provided that $c$ is chosen to be sufficiently large. Indeed, letting $\theta>0$ denote the minimal angle between any three (necessarily non-collinear) points of $S\cap [-N,N]^2$, i.e., $$\theta=\min \left\{\measuredangle P_1P_2P_3 : P_1,P_2,P_3\in S\cap [-N,N]^2 \text{ and } 0 < \measuredangle  P_1P_2P_3 <\pi \right\},$$ then, by \cite[Observation 1]{Solymosi2003}, we have \begin{equation}\label{eq:theta}
    \theta\gg \frac{1}{N}.
\end{equation}  
As the diameter of $[-N,N]^2$ is $\ll N$, the intersection of $C$ with $[-N,N]^2$ must be contained in an arc $A$ of $C$ of length $\ll N$. Therefore, the central angle $\alpha$  of $A$ in $C$ satisfies $\alpha r\ll N$, which gives 
\begin{equation}\label{eq:alpha}
    \alpha \ll \frac{N}{r}\leq\frac{1}{cN}.
\end{equation} 
Fixing $c$ sufficiently large, it follows from~\eqref{eq:theta} that $\frac{1}{cN}$ is strictly smaller than $\theta$, and thus, by \eqref{eq:alpha}, that 
\begin{equation}\label{eq:alphatheta}
    \alpha<\theta. 
\end{equation}
Now, suppose to the contrary that $A$ intersects $S\cap [-N,N]^2$ in a set containing three distinct points $P_1,P_2,P_3$; reordering the points if needed, we can assume that $0< \measuredangle P_1 P_2 P_3 \leq \frac{\alpha}{2}$. Thus, from the definition of $\theta$, we must have $\theta  \leq \frac{\alpha}{2}$,
contradicting \eqref{eq:alphatheta}.
We conclude that $A$ can intersect $S\cap [-N,N]^2$ in at most two points, completing the proof in this case as well.
\end{proof}

It remains to prove Proposition \ref{prop:mainalt}. We will begin by showing that one can assume the radius of $C$ is of a certain special form.
\begin{lemma}\label{lem:specialradius}
    There exists an absolute constant $K>0$ such that
    \begin{equation}\label{eq:specialradius}
      \max_{r\in[1,N]}c(r)\leq\max_{\substack{n,D\in\mathbf{N} \\
          (n,D)=1\\D\text{ squarefree} \\ n\leq KN^3\text{ and }D\leq KN^4}}c\left(\frac{n}{2\sqrt{D}}\right)
    \end{equation}
\end{lemma}
\begin{proof}
 Suppose that $C_r\subset\mathbf{R}^2$ is the circle of radius $r\geq 1$ centered at the
    origin. If the largest integer distance set contained in $C_r$ has size $2$, then the desired inequality automatically holds by taking $K\geq 2$ (since one can inscribe an equilateral
    triangle with side length $1$ inside a circle of radius $\frac{1}{\sqrt{3}}$). So, suppose that $C_r$ contains three distinct points $P_1,P_2,$ and $P_3$ such that
    \begin{equation*}
        d:=\|P_1-P_2\|,\qquad d':=\|P_1-P_3\|,\qquad\text{and}\qquad d'':=\|P_2-P_3\|
    \end{equation*}
    are all integers. Then, by Heron's formula and the circumradius formula, we have
    \begin{equation*}
        r=\frac{dd'd''}{\sqrt{(d+d'+d'')(d+d'-d'')(d-d'+d'')(-d+d'+d'')}}.
    \end{equation*}
    Thus, there exist $\ell_1,\ell_2\in\mathbf{N}$ and a squarefree $D\in\mathbf{N}$ such
    that $r=\frac{\ell_1}{\ell_2\sqrt{D}}$. Since $r\in [1,N]$, we certainly have
    $\ell_1\ll N^3$, $\ell_2\ll N^2$, and $D\ll N^4$. By dilating $C_r$ by $\ell_2$, we see that
    \[
      c(r)\leq c(\ell_2r)=c\left(\frac{\ell_3}{2\sqrt{D}}\right)
    \]
    for the positive integer $\ell_3=2\ell_1\ll N^3$. Note that the proof of~\eqref{eq:specialradius} is not yet complete, as $\ell_3$ and $D$ are not necessarily coprime. Since $D$ is squarefree, we can write
    $$\frac{\ell_3}{2\sqrt{D}}=\frac{\ell\sqrt{D_1}}{2\sqrt{D_2}}$$ for
    $\ell,D_1,D_2\in\mathbf{N}$ with $(\ell D_1,D_2)=1$ and both $D_1$ and $D_2$ being squarefree,
    where we have $D_1D_2= D$ and $\ell\leq \ell_3\ll N^3$ (specifically, $D_1$ is the product of those prime factors of $D$ that divide $\ell_3$). Next, by adapting
    the proof of Proposition~1 of~\cite{Bat-Ochir2019}, we will show that
    \begin{equation}\label{eq:flip}
      c\left(\frac{k\sqrt{m_1}}{2\sqrt{m_2}}\right) 
      \leq c\left(\frac{k}{2\sqrt{m_1m_2}}\right)
    \end{equation}
    whenever $k,m_1,m_2\in\mathbf{N}$ are such that $(km_1,m_2)=1$, $m_1$ and $m_2$ are
    squarefree, and $c\left(\frac{k\sqrt{m_1}}{2\sqrt{m_2}}\right)\geq 3$. The conclusion
    of the lemma will then follow by repeated applications of \eqref{eq:flip}, which eventually gives that
    \begin{equation*}
      c\left(\frac{\ell_3}{2\sqrt{D}}\right)=c\left(\frac{\ell\sqrt{D_1}}{2\sqrt{D_2}}\right)
      \leq
      c\left(\frac{n}{2\sqrt{D}}\right),
    \end{equation*}
    where
    \begin{equation*}
      n:=\prod_{\substack{p^a \mid\mspace{-2mu}\mid \ell_3 \\ (p,D)=1}}p^a.
    \end{equation*}

    It therefore remains to prove \eqref{eq:flip}. For $r'=\frac{k\sqrt{m_1}}{2\sqrt{m_2}}$ with $c(r')\geq 3$,  let
    $P_1',P_2',P_3'\in C_{r'}$ be any distinct points such that
    $$e:=\|P_1'-P_2'\|,\;\;\;e':=\|P_1'-P_3'\|,\;\;\;e'':=\|P_2'-P_3'\|\in\mathbf{N}$$ and let $O$ denote the origin. 
    We will show that all $e, e', e''$ are divisible by $m_1$ (which will then complete the proof by rescaling). Let us focus on $e$ first. By
    rotating and reflecting $C_{r'}$, we can assume that $P_1'$ lies on the
    positive $x$-axis, and that
    $$0<\measuredangle P_1'OP_2'=: 2\alpha\leq \pi,$$
    see Figure \ref{fig:lem52}. 
    \begin{figure}
    \centering
\centerline{\includegraphics[height=2.7in]{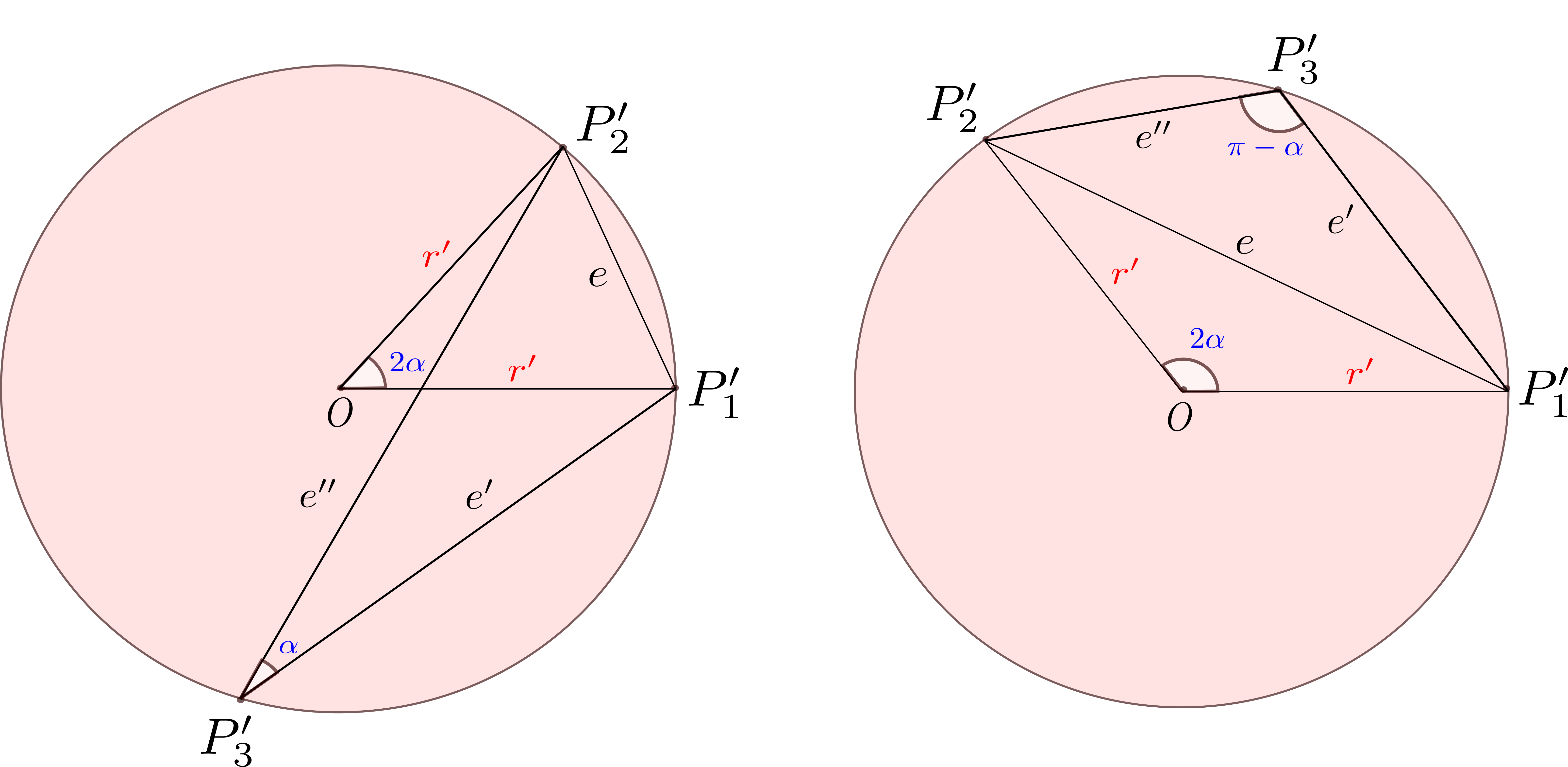}}
    \caption{The side lengths $e,e',e''$ of the triangle with vertices $P_1', P_2', P_3'$ on the circle $C_{r'}$ are integers.}
    \label{fig:lem52}
\end{figure}
If $\measuredangle P_1'OP_2'=\pi$, then $2r'=e\in\mathbf{N}$, which means that $m_1=m_2=1$ and there is nothing to prove. So, suppose that $2\alpha<\pi$. Then $\measuredangle P_1'P_3'P_2'\in\{\alpha,\pi-\alpha\}$, and
    hence, by the law of cosines, since all of
    the side lengths $e,e',e''$ of the triangle with vertices $P_1',P_2',P_3'$ are positive
    integers, we deduce that $\cos\measuredangle P_1'P_3'P_2'$ is rational. In addition, from the law of sines in the same triangle, we have that
    \begin{equation*}
    \sin \measuredangle P_1'P_3'P_2' =\sin \alpha =\frac{e}{2r'}= \frac{e\sqrt{m_2}}{k\sqrt{m_1}}.
    \end{equation*}
    Now, as
    \begin{equation*}
       \cos\measuredangle P_1'P_3'P_2' \in\left\{\pm\sqrt{1-\sin^2\alpha}\right\}=\left\{\pm\frac{\sqrt{k^2m_1-e^2m_2}}{k\sqrt{m_1}}\right\}
     \end{equation*}
     is rational and $0<\alpha<\frac{\pi}{2}$, we deduce that there must exist $m,t\in \mathbf{N}$ with $(m,t)=1$ such that
    \begin{equation*}
         \sqrt{k^2m_1-e^2m_2}=\frac{m}{t}\sqrt{m_1}.
     \end{equation*}
      Thus, $t^2e^2m_2=t^2k^2m_1-m^2m_1$, which implies that $t^2$ divides $m_1$. Since $m_1$ is squarefree, it follows that $t=1$, and so  $e^2m_2=k^2m_1-m^2m_1$. Therefore,  $e$ is divisible by $m_1$, since $m_1$ is
     squarefree and $(m_1,m_2)=1$.
     
     Now by rotating and reflecting $C_{r'}$ again, we can assume that $P'_3$ (respectively, $P_2'$) lies on the positive $x$-axis and that $0<\measuredangle  P_3' 
O P_1'=:2\alpha'\leq \pi$ (respectively, $0<\measuredangle  P_2' O P_3' =:2\alpha''\leq
\pi$). By applying the same argument as above we obtain that $e'$ (respectively, $e''$)
must also be divisible by $m_1$. Therefore, the distances between the points
     $\frac{1}{m_1}P_1',\frac{1}{m_1}P_2',\frac{1}{m_1}P_3'$ on the circle $C_{\frac{r'}{m_1}}$ are
     also integers.  This gives~\eqref{eq:flip}, as desired.
 \end{proof}

We will now prove Proposition~\ref{prop:main}. When the radius of $C$ has the special form $\frac{n}{2\sqrt{D}}$ with $n,D\in\mathbf{N}$, $D$ squarefree, and $\mathbf{Q}(\sqrt{-D})$ having class number $1$, this result already follows from work of Bat-Ochir~\cite{Bat-Ochir2019}. A minor modification of Bat-Ochir's argument, using a classical result on binary quadratic forms of negative discriminant in place of unique factorization, proves the proposition in general.

 \begin{lemma}\label{lem:circlebound}
     There exists an absolute constant $K>0$ such that, for all $n,D\in\mathbf{N}$ with $(n,D)=1$ and $D$ squarefree, we have
     \begin{equation*}
         c\left(\frac{n}{2\sqrt{D}}\right)\ll n^{\frac{K}{\log\log{n}}}.
     \end{equation*}
 \end{lemma}
\begin{proof}
    Using that $(n,D)=1$, we factor $n=n_1n_2$, where
    \begin{equation*}
        p\mid n_1\implies\left(\frac{-D}{p}\right)=-1\qquad\text{and}\qquad p\mid n_2\implies\left(\frac{-D}{p}\right)=1.
    \end{equation*}
    Let $C_r\subset\mathbf{R}^2$ denote the circle of radius $r:=\frac{n}{2\sqrt{D}}$ centered at the origin $O$, and suppose that $S\subset C_r$ is an integer distance set of maximal size $c(r)$, where we may assume, without loss of generality, that $c(r)\geq 3$. We may also assume, without loss of generality, that $S$ contains a point, $P_0$, that lies on the $x$-axis. Let $P,P'\in S$ be any two other distinct points, and set $2\alpha:=\measuredangle P_0OP$ and $a:=\|P-P_0\|\in\mathbf{Z}$. By reflecting about the $x$-axis, if needed, we may assume that $0< 2\alpha \leq \pi$. Then, as in the proof of Lemma~\ref{lem:specialradius}, $\cos{\measuredangle P_0P'P}$ must be rational by the law of cosines, since $\measuredangle P_0P'P\in\{\alpha,\pi-\alpha\}$ and the side lengths of the triangle with vertices $P_0,P',P$ are all integers.  On the other hand, by the law of sines, $\sin\measuredangle P_0P'P=\frac{a}{2r}$, so that  $$\cos{\measuredangle P_0P'P}\in\left\{\pm\frac{\sqrt{4r^2-a^2}}{2r}\right\}=\left\{\pm \frac{\sqrt{n^2-Da^2}}{n}\right\}$$  is rational. Thus, $\sqrt{n^2-Da^2}$ must be rational, and thus an integer, say $\sqrt{n^2-Da^2}=m\in\mathbf{Z}$. Squaring and rearranging to get
    \begin{equation*}
        -Da^2=m^2-n^2=m^2-(n_1n_2)^2,
    \end{equation*}
    it follows that both $a$ and $m$ are divisible by $n_1$, for otherwise $-D$ would be a quadratic residue modulo some prime dividing $n_1$, contradicting the definition of $n_1$. Setting $x=\frac{m}{n_1}$ and $y=\frac{a}{n_1}$, so that $x,y\in\mathbf{Z}$, we thus get
    \begin{equation*}
        x^2+Dy^2=n_2^2,\qquad \sin{\measuredangle P_0P'P}=\frac{y\sqrt{D}}{n_2},\qquad\text{and}\qquad\cos{\measuredangle P_0P'P}\in\left\{\pm\frac{x}{n_2}\right\},
    \end{equation*}
    from which it follows that, since the point $P$ is completely determined by the angle $2\alpha$ and the radius $r$, we must have
    \begin{equation*}
        |S|\leq \left|\left\{(x,y)\in\mathbf{Z}^2:x^2+Dy^2=n_2^2\right\}\right|.
    \end{equation*}
    The right-hand side above is well-known to be $\ll\tau(n_2^2)$ (see, for example,~\cite[Chapter 11]{Iwaniec1997}). Since $$\tau(n_2^2)\ll (n_2^2)^{O(1/\log\log(n_2^2))}\ll n^{O(1/\log\log{n})},$$ this completes the proof of the lemma.
\end{proof}

Combining Lemmas~\ref{lem:specialradius} and~\ref{lem:circlebound} now proves Proposition~\ref{prop:mainalt}, and thus Proposition~\ref{prop:main}.

\section{Conclusion of the argument}\label{sec:end}

We will need two final preparatory lemmas to prove Theorem~\ref{thm:main}.

\begin{lemma}\label{lem:projection}
    Let $\mathcal{C}_k\subset\mathbf{C}^{k+2}$ be an affine curve of degree $d$, and let $\pi:\mathbf{C}^{k+2}\to\mathbf{C}^{2}$ denote the projection map $\pi(x,y,d_1,\dots,d_k)=(x,y)$. Then $\pi(\mathcal{C}_k)$ is contained in the union of at most $d$ irreducible affine curves of degree at most $d$.
\end{lemma}
\begin{proof}
By Chevalley's theorem (see, for example, Exercise 3.19 of~\cite[Chapter I]{Hartshorne77}), there exist a finite number of affine varieties $V_1,V_1',\dots,V_n,V'_n\subset\mathbf{C}^2$ such that
\begin{equation*}
    \pi(\mathcal{C}_k)=\bigcup_{i=1}^n\left( V_i \setminus V_i'\right).
\end{equation*}
Note that we must have $\dim V_i\leq 1$ for all $i=1,\dots,n$, for otherwise we could find a line $L\subset\mathbf{C}^2$ that intersects $\pi(\mathcal{C}_k)$ in infinitely many points for which the hyperplane $\pi^{-1}(L)=L\times\mathbf{C}^{k}$ does not contain $\mathcal{C}_k$, which yields a contradiction since then $\pi^{-1}(L)\cap\mathcal{C}_k$ would consist of finitely many points while its projection under $\pi$ would consist of infinitely many points.

Thus, $\pi(\mathcal{C}_k)$ can be written as the union of a finite number of irreducible curves $C_1,\dots,C_{s}\subset\mathbf{C}^2$ and a finite number of points $z_1,\dots,z_{r}\in\mathbf{C}^2$, minus a finite number of points $w_{1},\dots,w_{t}\in\mathbf{C}^2$ that are distinct from $z_1,\dots,z_r$. A generic line $L\subset\mathbf{C}^2$ avoids $w_1,\dots,w_t$ and is distinct from any of the curves $C_1,\dots,C_{s}$ that happen to be a line, and thus intersects $\overline{\pi(\mathcal{C}_k)}$ in at most $|\overline{\mathcal{C}_k}\cap\overline{\pi^{-1}(L)}|\leq d$ points. Therefore, $\deg{\overline{\pi(\mathcal{C}_k)}}\leq d$, from which the conclusion of the lemma follows.
\end{proof}

\begin{lemma}\label{lem:fieldofdef}
    Let $C\subset\mathbf{C}^2$ be an irreducible affine curve of degree $d$ and $m\in\mathbf{N}$ be squarefree. Suppose that $C$ contains at least $d^2+1$ points in $\left\{(x,y\sqrt{m}):x,y\in\mathbf{Q}\right\}$. Then
    \begin{equation*}
        C'=\left\{(x,y)\in\mathbf{C}^2:(x,y\sqrt{m})\in C\right\}
    \end{equation*}
    is an irreducible affine curve of degree $d$ defined over $\mathbf{Q}$.
\end{lemma}
\begin{proof}
    Since the curve $C$ is irreducible of degree $d$, it is the zero set of some
    irreducible polynomial $f\in\mathbf{C}[x,y]$ of degree $d$. Then, $C'$ is the zero set
    of the polynomial $g(x,y):=f(x,y\sqrt{m})$, which is also irreducible of degree
    $d$. We will show that there exists a non-zero $\alpha\in\mathbf{C}$ such that $\alpha
    g\in\mathbf{Q}[x,y]$. It will then follow that $C'$ is an irreducible affine curve of
    degree $d$ defined over $\mathbf{Q}$, since $\alpha g$ is also irreducible, has degree $d$, and has the same zero set as $g$ in $\mathbf{C}^2$.     
   
    So, fix a set $\mathcal{P}$ of $n:=d^2+1$  rational points $z_1,\ldots,z_{n}$ on $C'$. Consider the general polynomial
    \begin{equation*}
        Q(x,y)=\sum_{(a,b)\in\mathbf{N}^2:\; a+b\leq d}\;c_{(a,b)} x^{a}y^{b}\in\mathbf{C}[x,y]
    \end{equation*}
    of degree at most $d$ subject to the conditions
    \begin{equation}\label{eq: linear system}
        Q(z_1)=\cdots=Q(z_{n})=0.
    \end{equation}
    Note that \eqref{eq: linear system} is a linear system of $n$ equations and
    $\binom{d+2}{2}$ unknowns, and thus of the form $Mv=0\in\mathbf{C}^{n}$ for a certain
    $n\times \binom{d+2}{2}$ matrix $M$ with rational entries. Since the non-zero
    polynomial $Q=g$ has degree at most $d$ and vanishes on
    $\mathcal{P}\subset C'$, the system \eqref{eq: linear system} has a non-trivial
    solution in $\mathbf{C}^{\binom{d+2}{2}}$. Equivalently, the linear map
    $T_{\mathbf{C}}:\mathbf{C}^{\binom{d+2}{2}}\rightarrow\mathbf{C}^{n}$ with
    $T_{\mathbf{C}}v=Mv$ has $\dim \ker T_{\mathbf{C}}\geq 1$. Since $M$ has rational entries, the map $T_{\mathbf{Q}}:\mathbf{Q}^{\binom{d+2}{2}}\rightarrow\mathbf{Q}^m$ with $T_{\mathbf{Q}}v=Mv$ is well-defined, and also has $\dim\ker T_\mathbf{Q}=\dim\ker T_\mathbf{C}\geq 1$. Indeed, the rank of $M$ is the largest order of a non-vanishing minor of $M$, and is thus independent of whether $M$ is viewed over $\mathbf{C}$ or over $\mathbf{Q}$.
    
    By the above, the linear system $Mv=0$ in $\mathbf{Q}^{n}$ has a non-trivial solution in
    $\mathbf{Q}^{\binom{d+2}{2}}$, i.e. there exists a non-zero polynomial
    $Q\in\mathbf{Q}[x,y]$, of degree at most $d$, vanishing on $\mathcal{P}$. Both $g$ and
    $Q$ vanish on $\mathcal{P}$, and thus on at least $d^2+1\geq \deg{Q}\cdot\deg{g}+1$
    points of $\mathbf{C}^2$. By B\'ezout's theorem, $g$ and $Q$ have a common
    factor. Since $g$ is irreducible, $g$ divides $Q$; and since $\deg Q\leq d=\deg g$, we
    conclude that $Q$ is a non-zero, constant multiple of $g$.
\end{proof}

Now we are ready to prove Theorem~\ref{thm:main}, from which Corollaries~\ref{cor:mainno4}
and~\ref{cor:mainno3} immediately follow (the latter by invoking Propositions~\ref{prop:lines} and~\ref{prop:main}). 

\begin{proof}[Proof od Theorem~\ref{thm:main}]
  We follow the outline in Section~\ref{sec:setup}. We will assume that $N$ is larger than any specified absolute constant, since the results are trivial for bounded $N$. Let $S\subset[-N,N]^2$ be an integer distance set and let $k=k(N)\in\mathbf{N}$, with $k(N)\to\infty$ as
  $N\to\infty$ and $k(N)\ll\log\log{N}$, be a parameter to be chosen later. We may assume,
  without loss of generality, that $S$ contains the origin and at least $k$
  other points, and, further, that there exists an $M\in\mathbf{N}$ and a squarefree $m\in\mathbf{N}$ such that
  \begin{equation*}
    S\subset \left\{(x,y\sqrt{m}):x,y\in\frac{1}{2M}\mathbf{Z}\right\},
  \end{equation*}
  where $M\ll N$. Pick $k$ distinct points $\tilde{P}_1,\dots,\tilde{P}_k\in S$, none of which is the
  origin, so that each $\tilde{P}_j=(a_j,\sqrt{m}b_j)$ with $a_j,b_j\in\mathbf{Q}$. Set $P_j=(a_j,b_j)$ for $j=1,\dots,k$. Using the notation of Section~\ref{sec:ag}, consider the affine variety
  \begin{equation*}
    X_k:=\left\{(x,y,d_1,\dots,d_k):Q_{m,P_j}(x,y,d_j)=0\text{ for }j=1,\dots,k\right\}
  \end{equation*}
  in $\mathbf{C}^{k+2}$. Each point $\tilde{P}=(x,y\sqrt{m})\in S$ corresponds to the rational
  point 
  \begin{equation*}
    \left[x:y:1:\|\tilde{P}-\tilde{P}_1\|:\cdots:\|\tilde{P}-\tilde{P}_k\|\right]\in \overline{X_k}  
  \end{equation*}
    of height $\ll N^2$. By
  Lemma~\ref{lem:generalX}, $\overline{X_k}$ is an irreducible surface of degree
  $2^k$. Thus, it follows from Theorem~\ref{thm:surfacecount} and
  Lemma~\ref{lem:projection} that $S$ is contained in the union of $\tau\ll
  e^{O(k)}N^{3/2^{k/2+1}}$ irreducible affine curves $C_1,\dots,C_{\tau}$, each of degree at most $\sigma\ll
  e^{O(k)}N^{3/2^{k/2+1}}$.

  For each $j=1,\dots,\tau$, either $C_j$ contains at least $(\deg{C_j})^2+1$ points in $S$, in which
  case $\left\{(x,y)\in\mathbf{C}^2:(x,y\sqrt{m})\in C_j\right\}$ is defined over $\mathbf{Q}$ by Lemma~\ref{lem:fieldofdef}, or else $S\cap C_j$
  contains at most $(\deg C_j)^2\ll e^{O(k)}N^{3/2^{k/2}}$ points. It follows that all but $\ll e^{O(k)}N^{5/2^{k/2}}$ points of $S$ are contained in the union of $\tau'\ll e^{O(k)}N^{3/2^{k/2+1}}$
  irreducible affine curves $\tilde{C}_1,\dots,\tilde{C}_{\tau'}$ of degree at most $\sigma$, where each
  \begin{equation*}
    \tilde{C}_j':=\left\{(x,y)\in\mathbf{C}^2:(x,y\sqrt{m})\in \tilde{C}_j\right\}
  \end{equation*}
  is defined over $\mathbf{Q}$.

  Let $1\leq j\leq \tau'$. Assume first that $\tilde{C}_j$ is not a line or
  circle. We will now bound $|S\cap\tilde{C}_j|$. We may assume, without loss of
  generality (by applying a translation), that $S\cap\tilde{C}_j$ contains the origin
  and that the origin is not a singular point of $\tilde{C}'_j$. Recall
  that an irreducible affine curve in $\mathbf{C}^2$ of degree $d$ can have at most
  $\frac{(d-1)(d-2)}{2}$ singular points (see, for example, Exercise 3 of~\cite[Chapter 3
  \S 2]{Shafarevich2013}). Denoting the set of singular points of $\tilde{C}_j'$ by $\mathcal{S}$,
  either $S\cap\tilde{C}_j$ contains at most
  \begin{equation}\label{eq:pointnumber}
   2+k(k-1)(\deg{\tilde{C}_j})^2+3\deg{\tilde{C}_j}(\deg{\tilde{C}_j}-2)
  \end{equation}
  nonsingular points of $\tilde{C}_j$, or else there exist $k$ distinct points
  \begin{equation*}
  \tilde{P}_1'=(a_1',\sqrt{m}b_1'),\dots, \tilde{P}_k'=(a_k',\sqrt{m}b_k')\in S\cap\tilde{C}_j,  
  \end{equation*}
  none of which is the origin or a singular point of $\tilde{C}_j$, such that
  \begin{equation*}
    \left\{\left(\frac{a_{n}'+a_{n'}'}{2}+i\sqrt{m}\frac{b_n'-b_{n'}'}{2},\frac{b_n'+b_{n'}'}{2}-i\frac{a_n'-a_{n'}'}{2\sqrt{m}}\right):1\leq
    n\neq n'\leq k\right\}\cap \tilde{C}'_j=\emptyset
  \end{equation*}
  and
  \begin{equation}\label{eq:pmtwice}
        \left\{s\pm i\sqrt{m}t:(s,t)\in\mathcal{S}\right\}\cap\left\{a'_n\pm i\sqrt{m} b'_n:n=1,\dots,k\right\}=\emptyset  
      \end{equation}
 and, if $\deg{\tilde{C}_j'}\geq 3$,
  \begin{equation}\label{eq:Rpm}
        \left\{(a_n,b_n):n=1,\dots,k\right\}\cap \bigcap_{l=0}^{\deg{\tilde{C}_j'}-2}V(\tilde{R}_{l,+},\tilde{R}_{l,-}) =\emptyset,
      \end{equation} where the $\tilde{R}_{l,\pm}$ are as in Lemma \ref{lem:generalC}.
      Indeed, the condition~\eqref{eq:pmtwice} forbids at most $$4|\mathcal{S}|\leq 2(\deg{\tilde{C}'_j}-1)(\deg{\tilde{C}'_j}-2)<2\deg{\tilde{C}'_j(\deg{\tilde{C}'_j}-2)}$$ points, and, if $\deg{\tilde{C}_j'}\geq 3$,  the condition \eqref{eq:Rpm} excludes at most $\deg{\tilde{C}'_j}(\deg{\tilde{C}_j'}-2)$ points. The reason for the latter is that $\bigcap_{l=0}^{\deg{\tilde{C}_j'}-2}V(\tilde{R}_{l,+},\tilde{R}_{l,-})$ is contained in a curve of degree at most $\deg{\tilde{C}_j'}-2$ (by Lemma~\ref{lem:Rpolys}), which thus has no irreducible components that agree with the irreducible curve $\tilde{C}_j'$; therefore, the intersection of $\tilde{C}_j'$ with $\bigcap_{l=0}^{\deg{\tilde{C}_j'}-2}V(\tilde{R}_{l,+},\tilde{R}_{l,-})$ can contain at most $\deg{\tilde{C}_j'}(\deg{\tilde{C}_j'}-2)$ points by B\'ezout's inequality. Moreover, if $\tilde{P}_1',\dots,\tilde{P}_{k'}'\in S\cap\tilde{C}_j$, with $k'<k$, have already been chosen so that
  \begin{equation*}
    \left\{\left(\frac{a_n'+a_{n'}'}{2}+i\sqrt{m}\frac{b_n'-b_{n'}'}{2},\frac{b_n'+b_{n'}'}{2}-i\frac{a_n'-a_{n'}'}{2\sqrt{m}}\right):1\leq
    n\neq n'\leq k'\right\}\cap \tilde{C}'_j=\emptyset,
  \end{equation*}
  then there are at most $2k'(\deg{\tilde{C}'_j})^2$ points $(a,\sqrt{m}b)\in \tilde{C}_j$ for which
  \begin{equation*}
    \left\{\left(\frac{a_n'+a}{2}\pm
        i\sqrt{m}\frac{b_n'-b}{2},\frac{b_n'+b}{2}\mp i\frac{a_n'-a}{2\sqrt{m}}\right):1\leq
      n\leq k'\right\}\cap \tilde{C}'_j\neq\emptyset.
  \end{equation*}
  This is because, for all such points, $(a,b)$ lies on one of the intersections
  \begin{equation*}
    \left\{(a,b)\in\mathbf{C}^2:\left(\frac{a_n'+a}{2}\pm
        i\sqrt{m}\frac{b_n'-b}{2},\frac{b_n'+b}{2}\mp i\frac{a_n'-a}{2\sqrt{m}}\right)\in\tilde{C}_j'\right\}\cap \tilde{C}'_j,\qquad n=1,\dots,k',  
  \end{equation*}
    each of which consists exactly of the points $(a,b)\in\tilde{C}_j'$ with $T_{\pm}(a,b)\in\tilde{C}'_j+c_{n,\pm}$, where $T_{\pm}$ is a linear transformation that maps $\mathbf{C}^2$ to the line $\ell_{\pm}=V(x\pm i\sqrt{m})$ and $c_{n,\pm}\in\mathbf{C}^2$. For each choice of $1\leq n\leq k'$ and sign $+$ or $-$, there are at most $\deg \tilde{C}_j'$ points $z$ in $\ell_{\pm}\cap (\tilde{C}_j'+c_{n,\pm})$ by B\'ezout's inequality 
     (Lemma \ref{lem:Bezout}), and, for each such $z$, there are at most $\deg\tilde{C}_j'$ points $(a,b)\in \tilde{C}_j'$ with $T_{\pm}(a,b)=z$, again by B\'ezout's inequality (since $T_{\pm}(a,b)=z$ gives the equation of a line in the variables $a$ and $b$, while $\tilde{C}_j'$ is not a line). Therefore, each of the above intersections can contain at most $(\deg{\tilde{C}'_j})^2$ points; by summing
  $$\sum_{k'=1}^{k-1}2k'(\deg{\tilde{C}'_j})^2=k(k-1)(\deg{\tilde{C}'_j})^2,$$ we conclude that as long as there are at least~\eqref{eq:pointnumber} nonsingular
  points of $\tilde{C}_j$ in $S\cap \tilde{C}_j$, we can find a suitable choice of $\tilde{P}_1'=(a_1',\sqrt{m}b_1'),\dots,\tilde{P}_k'=(a_k',\sqrt{m}b_k')$. We then set $P'_j=(a_j',b_j')$ for $j=1,\dots,k$ and  consider the affine variety
  \begin{equation*}
    \mathcal{C}_k:=\left\{(x,y,d_1,\dots,d_k):Q_{m,P_n'}(x,y,d_{n})=0\text{ for all }n=1,\dots,k\right\}\cap(\tilde{C}_j'\times\mathbf{C}^k)
  \end{equation*}
  in $\mathbf{C}^{k+2}$. Every point in $S\cap \tilde{C}_j$ corresponds to a rational point on $\overline{\mathcal{C}_k}$ of height $\ll N^2$. By Lemma~\ref{lem:generalC}, $\overline{\mathcal{C}_k}$ is an irreducible
  curve of degree at least $2^k$ and at most $2^k \sigma$, so that, by Theorem~\ref{thm:curvecount},
  $|S\cap\tilde{C}_j|\ll e^{O(k)}N^{O(2^{-k})}$ whenever $S\cap\tilde{C}_j$ contains at least~\eqref{eq:pointnumber} nonsingular points of $\tilde{C}_j$.  If $S\cap\tilde{C}_j$ contains fewer
  than~\eqref{eq:pointnumber} nonsingular points of $\tilde{C}_j$, then, as $N$ was chosen to be sufficiently large,
  $|S\cap\tilde{C}_j|\ll k^2\sigma^2\ll e^{O(k)}N^{O(2^{-k})}$.  Thus, either way, we have that
\begin{equation*}
  \left|S\cap\tilde{C}_j\right|\ll e^{O(k)}N^{O(2^{-k})}
\end{equation*}
whenever $\tilde{C}_j$ is not a line or circle. 

Now assume that $\tilde{C}_j$ is a line or circle. To bound $|S\cap\tilde{C}_j|$, we proceed as in the previous case, with a few modifications. We may again assume, without loss of generality, that $S\cap\tilde{C_j}$ contains the origin. Either $S$ contains at most $2+2k(k-1)$ points outside of $\tilde{C}_j$, or else there exist $k$ distinct points
\begin{equation*}
    \tilde{P}_1'=(a_1',\sqrt{m}b_1'),\dots, \tilde{P}_k'=(a_k',\sqrt{m}b_k')\in S\setminus \tilde{C}_j,
\end{equation*}
none of which is the center of $\tilde C_j$ in the case that $\tilde C_j$ is a circle, such that 
\begin{equation*}
    \left\{\left(\frac{a_{n}'+a_{n'}'}{2}+i\sqrt{m}\frac{b_n'-b_{n'}'}{2},\frac{b_n'+b_{n'}'}{2}-i\frac{a_n'-a_{n'}'}{2\sqrt{m}}\right):1\leq
    n\neq n'\leq k\right\}\cap \tilde{C}'_j=\emptyset.
\end{equation*}
Indeed, if $\tilde{P}_1',\dots,\tilde{P}_{k'}'\in S\setminus \tilde{C}_j$ with $k'<k$ have already been chosen so that
  \begin{equation*}
    \left\{\left(\frac{a_n'+a_{n'}'}{2}+i\sqrt{m}\frac{b_n'-b_{n'}'}{2},\frac{b_n'+b_{n'}'}{2}-i\frac{a_n'-a_{n'}'}{2\sqrt{m}}\right):1\leq
    n\neq n'\leq k'\right\}\cap \tilde{C}'_j=\emptyset,
  \end{equation*}
then there are at most $4k'$ points $(a,\sqrt{m}b)\in S\setminus\tilde{C}_j$ for which
  \begin{equation*}
    \left\{\left(\frac{a_n'+a}{2}\pm
        i\sqrt{m}\frac{b_n'-b}{2},\frac{b_n'+b}{2}\mp i\frac{a_n'-a}{2\sqrt{m}}\right):1\leq
      n\leq k'\right\}\cap \tilde{C}'_j\neq\emptyset,
  \end{equation*}
  since each of the $k'$ sets
\begin{equation*}
    \left\{\left(\frac{a_n'+a}{2}\pm
        i\sqrt{m}\frac{b_n'-b}{2},\frac{b_n'+b}{2}\mp i\frac{a_n'-a}{2\sqrt{m}}\right):\; (a,b)\in \mathbf{C}^2\right\}
\end{equation*}
is the union of at most two lines of the form $V(x\pm i\sqrt{m}y)+c_{n,\pm}$ for some $c_{n,\pm}\in\mathbf{C}^2$; each such set intersects $\tilde{C}_j$ in at most $4$ points (by B\'ezout's inequality, as $\deg{\tilde{C}_j}\leq 2$), while at the same time there can only be one point $(a,b)$ for which $(a,\sqrt{m}b)\in S$ on each of the lines $V(x\pm i\sqrt{m}y)+c_{n,\pm}$  (as each such line contains at most one point in $\mathbf{R}^2$). By summing
  \begin{equation*}
      \sum_{k'=1}^{k-1}4k'=2k(k-1),
  \end{equation*}
  we conclude that as long as there are at least $2+2k(k-1)$ points outside of $\tilde{C}_j$ in $S$, we can find a suitable choice of $\tilde{P}_1'=(a_1',\sqrt{m}b_1'),\dots,\tilde{P}_k'=(a_k',\sqrt{m}b_k')$. We then set $P'_j=(a_j',b_j')$ for $j=1,\dots,k$ and, again, consider the affine variety
    \begin{equation*}
    \mathcal{C}_k:=\left\{(x,y,d_1,\dots,d_k):Q_{m,P_n'}(x,y,d_{n})=0\text{ for all }n=1,\dots,k\right\}\cap(\tilde{C}_j'\times\mathbf{C}^k).
  \end{equation*}
  By Lemma~\ref{lem:Clineorcirc}, $\overline{\mathcal{C}}_k$ is an irreducible curve of degree between $2^k$ and $2^{k+1}$. Therefore, by Theorem~\ref{thm:curvecount}, $|S\cap\tilde{C}_j|\ll e^{O(k)}N^{O(2^{-k})}$ whenever $S\setminus\tilde{C}_j$ contains at least $2+2k(k-1)$ points.

  Thus, if, for all $1\leq j\leq \tau'$, we have that either $\tilde{C}_j$ is not a line or circle or $\tilde{C}_j$ is a line or circle and $S\setminus\tilde{C}_j$ contains at least $2+2k(k-1)$ points, then
  \begin{equation*}
      |S|\ll \tau'e^{O(k)}N^{O(2^{-k})}\ll e^{O(k)}N^{O(2^{-k})}.
  \end{equation*}
  This means that there exists a constant $c>0$ for which
  \begin{equation*}
      |S|\ll e^{ck}N^{1/c^k}.
  \end{equation*}
   Picking $k=\lfloor\log_c\log{N}\rfloor$ produces the bound $|S|\ll(\log{N})^{O(1)}$. 
   
   Finally, if, for some $1\leq j\leq \tau'$, $\tilde{C}_j$ is a line or circle and $S\setminus\tilde{C}_j$ contains fewer than $2+2k(k-1)$ points, as $N$ is chosen to be sufficiently large, we certainly have that all but at most $\ll(\log\log{N})^2$ points of $S$ lie on a single line or circle. This completes the proof of the theorem.
\end{proof}

\bibliographystyle{plain}
\bibliography{bib}

\end{document}